\documentclass[a4paper]{amsart}

\RequirePackage{amsmath}
\RequirePackage{bm}
\RequirePackage{amssymb}
\RequirePackage{upref}
\RequirePackage{amsthm}
\RequirePackage{enumerate}
\RequirePackage{pb-diagram}
\RequirePackage{amsfonts}
\RequirePackage[mathscr]{eucal}
\RequirePackage{verbatim}
\RequirePackage{xr}
\RequirePackage{graphicx}
\usepackage{calc}
\usepackage{xspace}
\RequirePackage{color}
\RequirePackage{ifthen}

\newcommand{\ud}{\mathrm{d}}
\newcommand{\ska}[1]{\big\langle#1\big\rangle}
\newcommand{\invv}{\tilde{v}}
\newcommand{\normiii}[1]{{\left\vert\kern-0.25ex\left\vert\kern-0.25ex\left\vert #1
    \right\vert\kern-0.25ex\right\vert\kern-0.25ex\right\vert}}

\newcommand{\f}{\varphi}

\newcommand{\di}[1]{#1\nobreakdash-\hspace{0pt}dimensional}

\newcommand{\msp[1]}[1]{\mspace{#1mu}}

\newcommand{\R}[1][n+1]{{\protect\mathbb R}^{#1}}

\newcommand{\N}{{\protect\mathbb N}}

\newcommand{\eR}{\stackrel{\lower1ex \hbox{\rule{6.5pt}{0.5pt}}}{\msp[3]\R[]}}
\newcommand{\eN}{\stackrel{\lower1ex \hbox{\rule{6.5pt}{0.5pt}}}{\msp[1]\N}}
\newcommand{\eO}{\stackrel{\lower1ex \hbox{\rule{6pt}{0.5pt}}}{\msc O}}

\newcommand{\uu}{\cup}

\newcommand{\uuu}{\bigcup}

\newcommand{\uud}{ \stackrel{\lower 1ex \hbox {.}}{\uu}}
\newcommand{\uuud}[1]{ \stackrel{\lower 1ex \hbox {.}}{\uuu_{#1}}}

\newcommand{\sminus}[1][28]{\raise 0.#1ex\hbox{$\scriptstyle\setminus$}}

\newcommand{\abs}[1]{\lvert#1\rvert}

\newcommand{\norm}[1]{\lVert#1\rVert}

\newcommand{\msc}{\protect\mathscr}

\providecommand{\bysame}{\makebox[3em]{\hrulefill}\thinspace}

\newcommand{\bt}{\begin{thm}}
\newcommand{\bl}{\begin{lem}}
\newcommand{\bc}{\begin{cor}}
\newcommand{\bd}{\begin{definition}}
\newcommand{\bpp}{\begin{prop}}
\newcommand{\br}{\begin{rem}}
\newcommand{\bn}{\begin{note}}
\newcommand{\be}{\begin{ex}}
\newcommand{\bes}{\begin{exs}}
\newcommand{\bb}{\begin{example}}
\newcommand{\bbs}{\begin{examples}}
\newcommand{\ba}{\begin{axiom}}
\newcommand{\bas}{\begin{assumption}}

\newcommand{\et}{\end{thm}}
\newcommand{\el}{\end{lem}}
\newcommand{\ec}{\end{cor}}
\newcommand{\ed}{\end{definition}}
\newcommand{\epp}{\end{prop}}
\newcommand{\er}{\end{rem}}
\newcommand{\en}{\end{note}}
\newcommand{\ee}{\end{ex}}
\newcommand{\ees}{\end{exs}}
\newcommand{\eb}{\end{example}}
\newcommand{\ebs}{\end{examples}}
\newcommand{\ea}{\end{axiom}}
\newcommand{\eas}{\end{assumption}}

\newcommand{\bp}{\begin{proof}}
\newcommand{\ep}{\end{proof}}
\newcommand{\eps}{\renewcommand{\qed}{}\end{proof}}

\newcommand{\bal}{\begin{align}}

\newcommand{\bi}[1][1.]{\begin{enumerate}[\upshape #1]}
\newcommand{\bia}[1][(1)]{\begin{enumerate}[\upshape #1]}
\newcommand{\bin}[1][1]{\begin{enumerate}[\upshape\bfseries #1]}
\newcommand{\bir}[1][(i)]{\begin{enumerate}[\upshape #1]}
\newcommand{\bic}[1][(i)]{\begin{enumerate}[\upshape\hspace{2\cma}#1]}
\newcommand{\bis}[2][1.]{\begin{enumerate}[\upshape\hspace{#2\parindent}#1]}
\newcommand{\ei}{\end{enumerate}}

\newcommand\ndots{\raise 0.47ex \hbox {,}\hskip0.06em\cdots %
     \raise 0.47ex \hbox {,}\hskip0.06em}

\newcommand\nd{\noindent}

\newskip\Csmallskipamount
\Csmallskipamount=\smallskipamount
\newskip\Cmedskipamount
\Cmedskipamount=\medskipamount
\newskip\Cbigskipamount
\Cbigskipamount=\bigskipamount

\newcommand\cvs{\vspace\Csmallskipamount}
\newcommand\cvm{\vspace\Cmedskipamount}

\newskip\csa
\csa=\smallskipamount

\newskip\cma
\cma=\medskipamount

\newskip\cba
\cba=\bigskipamount

\newdimen\spt
\newcommand\citem{\cvs\advance\itemno by
1{(\romannumeral\the\itemno})\hskip3pt}
\newcommand{\bitem}{\cvm\nd\advance\itemno by
1{\bf\the\itemno}\hspace{\cma}}

\newcount\itemno
\newskip\thmskip
\thmskip=\parindent

\newskip\hsk
\newenvironment{hinw}{\labelsep=0pt\begin{list}{}{\labelsep=0pt\itemindent=0pt\labelwidth=0pt\leftmargin=\parindent\rightmargin=0pt\partopsep=\cba}%
\item\it\nopagebreak\nopagebreak}%
{\end{list}}

\newcommand\bh{\begin{hinw}}
\newcommand{\eh}{\end{hinw}}

\newtheoremstyle{normal}
  {\cba}
  {\cba}
  {}
  {\thmskip}
  {\bfseries}
  {.}
  {\hsk}
  {}

\newtheoremstyle{abschnitt}
  {\cba}
  {\cba}
  {}
  {\thmskip}
  {\bfseries}
  {.}
  {\hsk}
  {}

\newtheoremstyle{italic}
  {\cba}
  {\cba}
  {\itshape}
  {\thmskip}
  {\bfseries}
  {.}
  {\hsk}
  {}

\newtheoremstyle{aufgaben}
  {\cba}
  {\cba}
  {}
  {}
  {\normalsize\bfseries}
  {.}
  {\hsk}
  {}

\newtheoremstyle{break}
  {\cba}
  {\cba}
  {\itshape}
  {}
  {\bfseries}
  {.}
  {\newline}
  {}

\swapnumbers
\theoremstyle{italic}
\newtheorem{thm}[subsection]{Theorem}
\newtheorem{lem}[subsection]{Lemma}
\newtheorem{prop}[subsection]{Proposition}
\newtheorem{cor}[subsection]{Corollary}

\theoremstyle{normal}
\newtheorem{rem}[subsection]{Remark}
\newtheorem{definition}[subsection]{Definition}
\newtheorem{example}[subsection]{Example}
\newtheorem{examples}[subsection]{Examples}
\newtheorem{ex}[subsection]{Exercise}
\newtheorem{note}[subsection]{}
\newtheorem{axiom}[subsection]{Axiom}
\newtheorem{assumption}[subsection]{Assumption}

\theoremstyle{aufgaben}
\newtheorem{exs}[subsection]{Exercises}

\swapnumbers

\numberwithin{equation}{section}
\numberwithin{figure}{section}

\newenvironment{textequation}[1][0.8]
{\begin{equation}
\begin{aligned}
\begin{minipage}{#1\linewidth}}
{\end{minipage}
\end{aligned}
\end{equation}
\ignorespacesafterend}

\newcommand{\btext}{\begin{textequation}}
\newcommand{\etext}{\end{textequation}}

\def\hinweis{\@startsection{subsection}{2}%
 \z@{0.7\linespacing\@plus 0.5\linespacing}{0.7\linespacing}%
{\normalfont\itshape\indent}}

\newcounter{hours}\newcounter{minutes}
\newcommand{\printtime}{%
\setcounter{hours}{\time/60}%
\setcounter{minutes}{\time-\value{hours}*60}%
\ifthenelse{\value{minutes}<10}{\thehours :0\theminutes}{\thehours:\theminutes}}

\usepackage[german,english]{babel}
\usepackage{graphicx}
\RequirePackage{amsmath}
\RequirePackage{bm}
\RequirePackage{amssymb}
\RequirePackage{upref}
\RequirePackage{amsthm}
\RequirePackage{enumerate}
\RequirePackage{pb-diagram}
\RequirePackage{amsfonts}
\RequirePackage[mathscr]{eucal}
\RequirePackage{verbatim}
\RequirePackage{xr}
\RequirePackage{graphicx}
\usepackage{calc}
\usepackage{xspace}
\usepackage{dsfont}

\makeatletter
\RequirePackage{color}
\newcommand{\ann}[1]{\renewcommand{\@makefnmark}{\mbox{$^{\color{red}{\@thefnmark}}$}}%
\footnote {#1}}
\makeatother

\RequirePackage{upref}
\RequirePackage{amsthm}
\RequirePackage{enumerate}
\usepackage[mathscr]{eucal}

\usepackage{xr-hyper}

\usepackage{calc}

\newlength{\oddsidemarginlength}
\newlength{\topmarginlength}

\newcounter{numberoflines}
\newcounter{tempcc}
\oddsidemargin=\oddsidemarginlength
\evensidemargin=\oddsidemargin
\topmargin=\topmarginlength
\usepackage[colorlinks=true,linkcolor=blue,citecolor=blue,urlcolor=blue]{hyperref}

\begin{document}

\flushbottom


\title{Dual flows in hyperbolic space and de Sitter space}

\author{Hao Yu}
\address{Ruprecht-Karls-Universit\"at, Institut f\"ur Angewandte Mathematik,
Im Neuenheimer Feld 294, 69120 Heidelberg, Germany}
\email{\href{mailto:h.yu@stud.uni-heidelberg.de}{h.yu@stud.uni-heidelberg.de}}
\thanks{This work has been supported by the Deutsche Forschungsgemeinschaft.}

%
\date{\today}
%


\begin{abstract}
We consider contracting flows in \di{($n+1$)} hyperbolic space and expanding flows in \di{($n+1$)} de Sitter space.
When the flow hypersurfaces are strictly convex we relate the contracting hypersurfaces and the expanding hypersurfaces by the Gau{\ss} map.
The contracting hypersurfaces shrink to a point $x_0$ in finite time while the expanding hypersurfaces converge to the maximal slice $\{ \tau =0\}$.
After rescaling, by the same scale factor, the resclaed contracting hypersurfaces converge to a unit geodesic sphere,
while the rescaled expanding hypersufaces converge to slice $\{ \tau = -1\}$ exponential fast in $C^\infty(\mathbb{S}^n)$.
\end{abstract}

\maketitle

\tableofcontents

\setcounter{section}{0}
\section{Introduction}
In a recent paper \cite{Ger13} a pair of dual flows was considered in $\mathbb{S}^{n+1}$. The one flow is the contracting flow
\begin{equation}
\label{direct-flow-Sn+1}
\dot{x} = -F \nu,
\end{equation}
while the other is an expanding flow
\begin{equation}
\label{inverse-flow-Sn+1}
\dot{x} = \tilde{F}^{-1} \nu ,
\end{equation}
where $F \in C^{\infty}(\varGamma_+)$ and $\tilde{F}$ is its inverse
\begin{equation}
\label{inverse-F}
\tilde{F} (\kappa_i) = \frac{1}{F(\kappa_i^{-1})}.
\end{equation}
There is a Gau{\ss} map for the pair $(\mathbb{S}^{n+1}, \mathbb{S}^{n+1})$, which maps closed, strictly convex hypersurfaces $M$ to their polar
sets $\tilde{M}$, cf. \cite[Chapter 9]{Ger06}. Gerhardt \cite{Ger13} proved, that the flow hypersurfaces of (\ref{direct-flow-Sn+1})
and (\ref{inverse-flow-Sn+1}) are polar sets of each other, if the initial hypersurface have this property. Under the assumption that $F$ is symmetric,
monotone, positive, homogeneous of degree 1, $F$ strictly concave (cf. \ref{definition-strictly-concavity}) and $\tilde{F}$ concave,
it is proved in \cite{Ger13} that the contracting flows contract to a round point and the expanding flows converge to an equator
such that after appropriate rescaling, both flows converge to a geodesic sphere exponential fast.

The Gau{\ss} map exists also for the pair $(\mathbb{H}^{n+1}, N)$, where $\mathbb{H}^{n+1}$ is the \di{$(n+1)$} hyperbolic space and $N$ is
the \di{$(n+1)$} de Sitter space, cf. \cite[Chapter 10]{Ger06}. We prove in this work similar results as in \cite{Ger13} by using this
duality.
Let $M(t)$ resp. $\tilde{M}(t)$ be solutions of the contracting flows
\begin{equation}
\label{direct-flow-Hn+1}
\dot{x} = -F \nu
\end{equation}
in $\mathbb{H}^{n+1}$ resp. the dual flows
\begin{equation}
\label{inverse-flow-N}
\dot{x} = - \tilde{F}^{-1} \nu
\end{equation}
in $N$, where $\tilde{F}$ is the inverse of $F$ defined by (\ref{inverse-F}).
We impose the following assumptions.
\begin{assumption}
\label{general-assmption}
Let $F \in C^{\infty}(\varGamma_+)$ be a symmetric, monotone, 1-homogeneous and concave curvature function satisfying the normalization
\begin{equation}
\label{normalization-condition}
F(1, \dots, 1)=1.
\end{equation}
We assume further, either
\begin{enumerate}
\item{$F$ is concave and $\tilde{F}$ is concave and the initial hypersurface $M_0$ is horoconvex (i.e. all principal curvatures $\kappa_i \geq 1$),}
\begin{description}
\hspace{-1.43cm} or
\end{description}
\item{$\tilde{F}$ is convex and $M_0$ is strictly convex.}
\end{enumerate}
\end{assumption}
We now state our main results
\begin{thm}
\label{main-results}
We consider curvature flows $(\ref{direct-flow-Hn+1})$ and $(\ref{inverse-flow-N})$ under assumption \ref{general-assmption} with initial smooth
hypersurfaces $M_0$ and $\tilde{M}_0$, where $\tilde{M}_0$ is the polar hypersurface of $M_0$. Then the both flows exist on the maximal time
interval $[0, T^*)$ with finite $T^*$. The hypersurfaces $\tilde{M}(t)$ are the polar hypersurfaces of $M(t)$ and vice versa during the evolution.
The contracting flow hypersurfaces in $\mathbb{H}^{n+1}$ shrink to a point $x_0$ while the expanding flow hypersurfaces in $N$ converge to a totally
geodesic hypersurface which is isometric to $\mathbb{S}^n$. We may assume the point $x_0$ is the Beltrami point by applying an isometry
such that the hypersurfaces of the expanding flow are all contained in $N_-$ and converge to the coordinate slice $\{ \tau =0 \}$.\\
Viewing $\mathbb{H}^{n+1}$ and $N$ as submanifolds of $\mathbb{R}^{n+1,1}$ and by introducing polar coordinates in the Euclidean part of
$\mathbb{R}^{n+1,1}$ centered in $(0, \dots, 0) \in \mathbb{R}^{n+1}$, we can write flow hypersurfaces in $\mathbb{H}^{n+1}$ resp. $N$ as graphs of
functions $u$ resp. $u^*$ over $\mathbb{S}^n$.
Let $\Theta = \Theta (t, T^*)$ be the solution of $(\ref{direct-flow-Hn+1})$ with spherical initial hypersurface and exitence
intervall $[0, T^*)$. Then the rescaled functions
\begin{equation}
\label{function-tilde-u}
\tilde{u} = u \Theta^{-1}
\end{equation}
and
\begin{equation}
\label{function-w}
w = u^* \Theta^{-1}
\end{equation}
are uniformly bounded in $C^\infty (\mathbb{S}^n)$. The rescaled principal curvatures $\kappa_i \Theta$ as well as $\tilde{\kappa}_i \Theta^{-1}$ are
uniformly positiv, where $\tilde{\kappa}_i$ are the principal curvatures of $\tilde{M}(t)$.\\
If the curvature function $F$ is further strictly concave or $F = \tfrac{1}{n} H$, then the rescaled functions $(\ref{function-tilde-u})$ resp.
$(\ref{function-w})$ converge to the constant functions $1$ resp. $-1$ in $C^\infty(\mathbb{S}^n)$ exponentially fast.
\end{thm}
Let us review some results concerning the contracting flows in $\mathbb{H}^{n+1}$. Under the assumption that the initial hypersurface
is strictly convex and satisfies the condition $\kappa_i H > n$ for each $i$, Huisken \cite{Hui86} proved that
the flow (\ref{direct-flow-Hn+1}) with $F= H$ converges in finite time to a round sphere.
Andrews \cite{And94-2} proved similar results for a general class of curvature function with argument $\kappa_i -1$.
Makowski \cite{Mak12} proved the contracting flow with a volume preserving term exists for all times and converges
to a geodesic sphere exponentially fast.\\
The key ingredient treating the contracting flow is the pinching estimates. Under assmuption \ref{general-assmption} (1) it follows by a similar
calculation as in \cite{Mak12}, while Gerhardt \cite{Ger15} proved the pinching estimates under assumption \ref{general-assmption} (2).\\
The elementary symmetric polynomials are defined by
\begin{equation}
H_k(\kappa_1, \dots, \kappa_n) = \sum_{1 \leq i_1< \cdots < i_k \leq n} \kappa_{i_1} \dots \kappa_{i_k}, \quad 1 \leq k \leq n.
\end{equation}
Examples of curvature functions $F$ satisfying assumption \ref{general-assmption} (1) (up to normalization condition (\ref{normalization-condition})) are
\begin{itemize}
\label{concave-and-inverse-concave-functions}
\item{the power means $\left( \tfrac{1}{n} \sum_{i} \kappa_i^r \right)^{1/r}$ for $\abs{r} \leq 1$, }
\item{$\sigma_k = H_k^{1/k}$ for $1 \leq k \leq n$,}
\item{the inverse $\tilde{\sigma}_k$ of $\sigma_k$ for $1 \leq k \leq n$, }
\item{$\left( H_k/H_l\right)^{1/(k-l)}$ for $0 \leq l < k \leq n$,}
\item{$H_n^{\alpha_n} H_{n-1}^{\alpha_{n-1} - \alpha_n} \cdots H_2^{\alpha_2 - \alpha_3} H_1^{\alpha_1 - \alpha_2}$ for $\alpha_i \geq 0$
and $\sum_i \alpha_i =1$.}
\end{itemize}
For a proof see \cite[Chapter 2]{And07}. Moreover, the curvature functions in the above list are all strictly concave with exception of the mean curvature
(cf. Section \ref{curvature-functions})\\
Examples of convex curvature functions $\tilde{F}$, which is used in assumption \ref{general-assmption} (2)
(up to normalization condition (\ref{normalization-condition})) are (cf. \cite[Remark 2.2.13]{Ger06})
\begin{itemize}
\item{the mean curvature $H$, }
\item{the length of the second fundamental form $\abs{A} = \left( \sum_i \kappa_i^2\right)^{1/2}$,}
\item{the complete symmetric functions\\ $\gamma_k (\kappa_1, \dots, \kappa_n) =
\left( \sum_{\abs{\alpha} = k} \kappa_1^{\alpha_1} \kappa_2^{\alpha_2} \dots \kappa_n^{\alpha_n}\right)^{1/k}$ for $1 \leq k \leq n$.}
\end{itemize}
Note that for convex $\tilde{F}$ under assumption \ref{general-assmption} (2), $F$ is of class $(K)$ and homogeneous of degree 1,
hence strictly concave. (cf. \cite[Definition 2.2.1, Lemma 2.2.12, 2.2.14]{Ger06}, \cite[Lemma 3.6]{Ger13})
\section{Setting and general facts}
\label{section:general-facts}
We now review some general facts about hypersurfaces from \cite[Chapter 1]{Ger06}.
Let $N$ be a \di{$(n+1)$} dimensional semi-Riemannian manifold and $M$ be a hypersurface in $N$. Geometric quantities in $N$ will be denoted
by $(\bar{g}_{\alpha \beta}), (\bar{R}_{\alpha \beta \gamma \delta}) $, etc., where greek indices range from $0$ to $n$. Quantities in $M$
will be denoted by $(g_{ij}), (h_{ij})$ etc., where latin indices range from $1$ to $n$. Generic coordinate systems in $N$ resp. $M$ will
be denoted by $(x^{\alpha})$ resp. $(\xi^i)$.

Covariant differentiation will usually be denoted by indices, only if ambiguities are possible, by a semicolon, e.g. $h_{ij;k}$.

Let $x: M \hookrightarrow N$ be a spacelike hypersurface (i.e. the induced metric is Riemannian) with a differentiable normal $\nu$, which is always
supposed to be normalized, and $(h_{ij})$ be the second fundamental form, and set $\sigma = \ska{\nu,\nu}$.\\
We have the \emph{Gau{\ss} formula}
\begin{equation}
x_{ij}^\alpha = - \sigma h_{ij} \nu^{\alpha},
\end{equation}
the \emph{Weingarten equation}
\begin{equation}
\nu_i^{\alpha} = h_i^k x_k^\alpha,
\end{equation}
the \emph{Codazzi equation}
\begin{equation}
h_{ij;k} - h_{ik;j} = \bar{R}_{\alpha \beta \gamma \delta} \nu^{\alpha} x_i^{\beta} x_j^{\gamma} x_k^{\delta},
\end{equation}
and the \emph{Gau{\ss} equation}
\begin{equation}
R_{ijkl} = \sigma \{ h_{ik} h_{jl} - h_{il} h_{jk}\} + \bar{R}_{\alpha \beta \gamma \delta} x_i^\alpha x_j^\beta x_k^\gamma x_l^\delta.
\end{equation}

Let us review some properties of $\mathbb{H}^{n+1}$ and $N$, cf. \cite[Section 10.2]{Ger06}.
We label the coordinates in the \di{$(n+2)$} Minkowski space $\mathbb{R}^{n+1,1}$ as $x = (x^a), 0 \leq a \leq n+1$, where $x^0$ is
the time function.
Recall that the hyperbolic space $\mathbb{H}^{n+1}$ and de Sitter space $N$ are the subspaces of  $\mathbb{R}^{n+1,1}$ defined by
\begin{equation}
\mathbb{H}^{n+1} = \{ x \in \mathbb{R}^{n+1,1}: \ska{x,x} = -1, x^0>0 \},
\end{equation}
\begin{equation}
N = \{ x \in \mathbb{R}^{n+1,1}: \ska{x,x} = 1\}.
\end{equation}
Introduce polar coordinates in the Euclidean part of $\mathbb{R}^{n+1,1}$ centered in $(0, \dots, 0) \in \mathbb{R}^{n+1}$ such that the
metric in $\mathbb{R}^{n+1,1}$ is expressed as
\begin{equation}
d \bar{s}^2 = - {dx^0}^2 + dr^2 + r^2 \sigma_{ij} d \xi^i d \xi^j,
\end{equation}
where $\sigma_{ij}$ is the spherical metric.

By viewing $\mathbb{H}^{n+1}$ as
\begin{equation}
\mathbb{H}^{n+1} =\{ (x^0, r, \xi^i): r = \sqrt{\abs{x^0}^2-1}, x^0 >0, \xi \in \mathbb{S}^n\},
\end{equation}
and by setting
\begin{equation}
\varrho = \textrm{arccosh} \, x^0,
\end{equation}
$\mathbb{H}^{n+1}$ has coordinates $(\varrho, \xi^i)$ and the metric
\begin{equation}
d \bar{s}_{\mathbb{H}^{n+1}}^2  = d \varrho^2 + \sinh^2 \varrho \, \sigma_{ij}\, d \xi^i d \xi^j.
\end{equation}
Similarly,
\begin{equation}
N = \{ (x^0, r, \xi^i): r = \sqrt{1 + \abs{x^0}^2}, x^0 \in \mathbb{R}, \xi \in \mathbb{S}^n\},
\end{equation}
and by setting the eigentime
\begin{equation}
\label{eigen-time}
\tau = \textrm{arcsinh} \, x^0,
\end{equation}
$N$ has coordinates $(\tau, \xi^i)$ and the metric
\begin{equation}
\label{eigen-time-coordinates}
d \bar{s}_{N}^2  = - d \tau^2 + \cosh^2 \tau \sigma_{ij} d \xi^i d \xi^j.
\end{equation}
\section{Strictly concave curvature functions}
\label{curvature-functions}
For $\xi, \kappa \in \mathbb{R}^n$, we write $\xi \sim \kappa$, if there is $\lambda \in \mathbb{R}$ such that $\xi = \lambda \kappa$.
\begin{definition}
\label{definition-strictly-concavity}
Let $F \in C^2(\varGamma)$ be a symmetric, monotone, 1-homogeneous and concave curvature function.
We call $F$ strictly concave (in non-radial directions), if
\begin{equation}
F_{ij} \xi^i \xi^j <0 \quad \forall \xi \not\sim \kappa \textrm{ and } \xi \not=0 ,
\end{equation}
or equivalently, if the multiplicity of the zero eigenvalue for $D^2 F(\kappa)$ is one for all $\kappa \in \varGamma$.
\end{definition}
Note since $F$ is homogeneous of degree $1$, $\kappa \in \varGamma$ is an eigenvector of $D^2 F(\kappa)$ with zero eigenvalue.
In \cite[Chapter 3]{Ger13} it is proved that $\sigma_k, \, 2 \leq k \leq n$ and the inverses
$\tilde{\sigma}_k$ of $\sigma_k$, $1 \leq k \leq n$ are strictly concave. In \cite[Chapter 2]{Hui99} it is proved that
$Q_k = H_{k+1}/H_k, 1 \leq k \leq n-1$ are strictly concave in $\varGamma_+$.
We consider the rest of the concave and inverse concave curvature functions listed on page \pageref{concave-and-inverse-concave-functions}.
\begin{lem}
The curvature functions
\begin{equation}
F = ( \tfrac{1}{n} \sum_{i} \kappa_i^r )^{1/r} \quad -1 \leq r <1
\end{equation}
are strictly concave in $\varGamma_+$.
\end{lem}
\begin{proof}
Note that $F$ converges locally uniformly to $\sigma_n = (\kappa_1 \cdots \kappa_n)^{1/n}$ as $r \rightarrow 0$ and $\sigma_n$ is strictly concave.
Furthermore, for $-1 \leq r <1$ and $r \not= 0$,
\begin{equation}
\frac{\partial F}{\partial \kappa^i} = n^{-1/r} \left(\sum_{l} \kappa_l^r \right)^{\frac{1}{r} - 1} \kappa_i^{r-1},
\end{equation}
\begin{equation}
\frac{\partial^2 F}{\partial \kappa^i \partial \kappa^j} = n^{-1/r} (1-r) \left(\sum_{l} \kappa_l^r \right)^{\frac{1}{r} - 2} \kappa_i^{r-2}
( \kappa_i \kappa_j^{r-1} - \sum_{l} \kappa_l^r \delta_{ij} ).
\end{equation}
Consider $\eta$ such that $F_{ij} \eta^j =0$. Since $r \not=1$,
\begin{equation}
\eta_i = \left( \sum_l \kappa_l^r \right)^{-1} \kappa_j^{r-1} \eta^j \kappa_i.
\end{equation}
Knowing that $F$ is concave for $|r| \leq 1$ we conclude that $F$ is strictly concave for $-1 \leq r <1$.
\end{proof}
\begin{lem}
Let $f^\alpha$ be concave in $\varGamma_+$ for all $1 \leq \alpha \leq k$
and strictly concave in $\varGamma_+$ for at least one index in $1 \leq \alpha \leq k$.
Let $\varphi$ be strictly monotone increasing and concave in $\varGamma_+$, then
\begin{equation}
F (\kappa_1, \cdots, \kappa_n) = \varphi(f^1(\kappa_1, \cdots, \kappa_n), \cdots, f^k(\kappa_1, \cdots, \kappa_n))
\end{equation}
is strictly concave in $\varGamma_+$.
\end{lem}
\begin{proof}
Let $0 \not= \xi \in \mathbb{R}^n$ and $\xi \not\sim \kappa$, then
\begin{equation}
F_{ij} \xi^i \xi^j = \varphi_\alpha f_{ij}^\alpha \xi^i \xi^j + \varphi_{\alpha \beta} f_i^\alpha f_j^\beta \xi^i \xi^j <0,
\end{equation}
since by assumption
\begin{equation}
\varphi_\alpha>0,  \quad \varphi_{\alpha \beta} \leq 0,  \quad f_{ij}^\alpha \xi^i \xi^j \leq 0
\end{equation}
and
\begin{equation}
\quad f_{ij}^\alpha \xi^i \xi^j<0 \textrm{ for at least one } 1 \leq \alpha \leq k.
\end{equation}
\end{proof}
Note that the weighted geometric mean
\begin{equation}
\varphi (f^1, \cdots, \f^k) = (f^1)^{\alpha_1} \cdots (f^k)^{\alpha_k} \quad \textrm{ with } \sum_{i} \alpha_i=1
\end{equation}
is a strictly monotone increasing and concave function.
Knowing that $H_{k+1}/H_k, 1 \leq k \leq n-1$ are strictly concave in $\varGamma_+$, we conclude that
\begin{equation}
\left( H_k/H_l\right)^{1/(k-l)}  = \left( H_{l+1}/ H_l \right)^{1/(k-l)} \cdots \left( H_{k}/ H_{k-1}\right)^{1/(k-l)}
\quad 0 \leq l < k \leq n
\end{equation}
and
\begin{equation}
H_n^{\alpha_n} H_{n-1}^{\alpha_{n-1} - \alpha_n} \cdots H_2^{\alpha_2 - \alpha_3} H_1^{\alpha_1 - \alpha_2} =
\left( \frac{H_1}{H_0}\right)^{\alpha_1} \left( \frac{H_2}{H_1}\right)^{\alpha_2} \cdots \left( \frac{H_n}{H_{n-1}}\right)^{\alpha_n}
\end{equation}
with $\alpha_i \geq 0$, $\sum_i \alpha_i =1$ and $\alpha_1 \not= 1$
are strictly concave in $\varGamma_+$.
\section{Polar sets and dual flows}
We state some facts about Gau{\ss} maps for $(\mathbb{H}^{n+1}, N)$, cf. \cite[Section 10.4]{Ger06}.
\begin{thm}
Let $x: M_0 \rightarrow M \subset \mathbb{H}^{n+1}$ be a closed, connected, strictly convex hypersurface.
Consider $M$ as a codimension $2$ immersed submanifold in $\mathbb{R}^{n+1,1}$ such that
\begin{equation}
x_{ij} = g_{ij} x - h_{ij} \tilde{x},
\end{equation}
where $\tilde{x} \in T_x (\mathbb{R}^{n+1,1})$ is the representation of the exterior normal vector
$\nu = (\nu^\alpha)$ of $M$ in $T_x(\mathbb{H}^{n+1})$.
Then the Gau{\ss} map
\begin{equation}
\tilde{x} : M_0 \rightarrow N
\end{equation}
is the embedding of a closed, spacelike, achronal, strictly convex hypersurface $\tilde{M} \subset N$.
Viewing $\tilde{M}$ as a codimension $2$ submanifold in $\mathbb{R}^{n+1,1}$, its Gaussian formula is
\begin{equation}
\tilde{x}_{ij} = - \tilde{g}_{ij} \tilde{x} + \tilde{h}_{ij} x,
\end{equation}
where $\tilde{g}_{ij},\tilde{h}_{ij}$ are the metric and second fundamental form of $\tilde{M}$ and $x$ is the embedding of $M$ which
also represents the future directed normal vector of $\tilde{M}$. The second fundamental form $\tilde{h}_{ij}$ is defined with respect to
the future directed normal vector, where the time orientation of $N$ is inherited from $\mathbb{R}^{n+1,1}$.
Furthermore, there holds
\begin{equation}
\tilde{h}_{ij} = h_{ij},
\end{equation}
\begin{equation}
\tilde{\kappa}_i = \kappa_i^{-1}.
\end{equation}
\qed
\end{thm}
We prove in the following that the duality is also valid in case of curvature flows.
\begin{lem}
\label{duality-result}
Let $\varPhi \in C^{\infty} (\mathbb{R}_+)$ be strictly monotone, $\dot{\varPhi} >0$, and let $F \in C^{\infty}(\varGamma_+)$ be a symmetric, monotone,
1-homogeneous curvature function such that $F |_{\varGamma_+}>0$ and such that the flows
\begin{equation}
\label{flow-eq-in-Hn+1-general-form}
\dot{x} = - \varPhi(F) \nu
\end{equation}
in $\mathbb{H}^{n+1}$ resp.
\begin{equation}
\dot{\tilde{x}} = - \varPhi(\tilde{F}^{-1}) \tilde{\nu}
\end{equation}
in $N$ with initial strictly convex hypersurfaces $M_0$ resp. $\tilde{M}_0$ exist on maximal time intervals $[0,T^*)$ resp.
$[0, \tilde{T}^*)$, where $\nu$ and $\tilde{\nu}$ are the exterior normal resp. past directed normal. The flow hypersurfaces
are then strictly conxex. Let $M(t)$ resp. $\tilde{M}(t)$ be the corresponding flow hypersurfaces, then $T^* = \tilde{T}^*$
and $M(t) = \tilde{M}(t)$ for all $t \in [0, T^*)$.
\end{lem}
\begin{proof}
The arguments are similar to those in \cite[Section 4]{Ger13} with combination with the results from \cite[Section 10.4]{Ger06}.
Since there holds
\begin{equation}
\ska{x,x} = 1,  \, \ska{\dot{x},x} = 0, \, \ska{x_j,x} =0, \, \ska{\tilde{x},x} =0,
\end{equation}
(see \cite[Lemma 10.4.1]{Ger06} for the last identity)
we can consider the flow (\ref{flow-eq-in-Hn+1-general-form}) as flow in $\mathbb{R}^{n+1,1}$
\begin{equation}
\dot{x} = - \varPhi \tilde{x},
\end{equation}
and we have the decomposition
\begin{equation}
\label{decomp-tanget-space}
T_x(\mathbb{R}^{n+1,1}) = T_x(\mathbb{H}^{n+1})  \, \oplus \, \ska{x}.
\end{equation}
Furthermore, we conclude from
\begin{equation}
\ska{\dot{\tilde{x}},x_j} = \varPhi_j, \, \ska{\dot{\tilde{x}}, \tilde{x}} = 0,  \, \ska{\dot{\tilde{x}},x} = \varPhi,
\end{equation}
from the Weingarten equation (see \cite[Lemma 10.4.3, 10.4.4]{Ger06})
\begin{equation}
x_j = \tilde{h}_j^k \tilde{x}_k,
\end{equation}
and from (\ref{decomp-tanget-space}) that
\begin{equation}
\dot{\tilde{x}} = \varPhi x + \varPhi^m x_m = \varPhi x + \varPhi^m \tilde{h}_m^k \tilde{x}_k,
\end{equation}
where
\begin{equation}
\varPhi^m = g^{mj} \varPhi_j,
\end{equation}
and the second fundamental form $\tilde{h}_{ij}$ is defined with respect to the future directed normal vector $\tilde{\nu}$.
The corresponding flow equation in $N$ has the form
\begin{equation}
\label{flow-eq-in-N}
\dot{\tilde{x}} = \varPhi \tilde{\nu} + \varPhi^m \tilde{h}_m^k \tilde{x}_k.
\end{equation}
Let $t_0 \in [0,T^*)$ and introduce polar coordinates in the Euclidean part of the Minkowski space as well as
an eigentime coordinate system in $N$ as in Section \ref{section:general-facts}.
For $\epsilon$ small and $t_0<t<t_0+ \epsilon$, $\tilde{M}(t)$ can be written as graph over $\mathbb{S}^n$
\begin{equation}
\tilde{M}(t) = \textrm{graph} \, \tilde{u}|_{\mathbb{S}^n},
\end{equation}
and we obtain the scalar flow equation
\begin{equation}
\frac{d \tilde{u}}{d t} = \varPhi \tilde{v}^{-1} + \varPhi^m \tilde{h}_m^{k} \tilde{u}_k,
\end{equation}
where
\begin{equation}
\invv^2 = 1 - \abs{D \tilde{u}}^2 = 1 - \frac{1}{\cosh^2 \tilde{u}} \sigma^{ij} \tilde{u}_i \tilde{u}_j.
\end{equation}
Note that $\tilde{\nu}$  in (\ref{flow-eq-in-N}) is the future directed normal
\begin{equation}
(\tilde{\nu}^\alpha) = \invv^{-1} (1,\check{\tilde{u}}^i),
\end{equation}
where
\begin{equation}
\check{\tilde{u}}^i = \frac{1}{\cosh^2 \tilde{u}} \sigma^{ij} \tilde{u}_j.
\end{equation}
Thus it holds in view of (\ref{flow-eq-in-N})
\begin{equation}
\begin{split}
\frac{\partial \tilde{u}}{\partial t} &= \frac{d \tilde{u}}{d t} - \tilde{u}_i \dot{\tilde{x}}^i\\
&= \varPhi \tilde{v}^{-1} + \varPhi^m h_m^{k} \tilde{u}_k - \varPhi \invv^{-1} \abs{D \tilde{u}}^2
- \varPhi^m \tilde{h}_m^k \delta_k^i \tilde{u}_i\\
&= \varPhi \invv.
\end{split}
\end{equation}
This is exactly the scalar curvature equation of the flow equation
\begin{equation}
\label{flow-eq-in-N-simplified}
\dot{\tilde{x}} = -\varPhi \tilde{\nu},
\end{equation}
where  $\tilde{\nu}$ in (\ref{flow-eq-in-N-simplified})
is the future directed normal and
\begin{equation}
\varPhi = \varPhi(F) = \varPhi(\tilde{F}^{-1}).
\end{equation}
Now $\tilde{h}_{ij}$ in $N$ is defined with respect to the future directed normal. By adapting the convention in \cite[p.307]{Ger06}
we switch the light cone in $N$ and by defining $\tau = - \mathrm{arcsinh} \, x^0$ in (\ref{eigen-time}) we still derive the flow
(\ref{flow-eq-in-N-simplified}) in $N$, where $\tilde{\nu}$ is now the past directed normal and the second fundamental form is defined
with respect to this normal.
The rest of the proof is identical to \cite[Theorem 4.2]{Ger13}.
\end{proof}

From now we shall employ this duality by choosing
\begin{equation}
\varPhi(r) = r.
\end{equation}

Note that the expanding flows in $\mathbb{H}^{n+1}$ was already considered in \cite{Ger11} and its non-scale-invariant version in \cite{Sch12}.

\section{Pinching estimates}
We consider the contracting flow
\begin{equation}
\label{flow-eq}
\begin{split}
&\dot{x} = - F \nu,\\
&x(0) = M_0
\end{split}
\end{equation}
in $\mathbb{H}^{n+1}$ with initial smooth and strictly convex hypersurfaces $M_0$, where $\nu$ is the exterior normal vector.

Under the assumptions of Theorem \ref{main-results} the curvature flow (\ref{flow-eq}) exists on a maximal time interval
$[0,T^*), 0 < T^* \leq \infty$, cf. \cite[Theorem 2.5.19, Lemma 2.6.1]{Ger06}.

\begin{thm}
\label{preservation-h-convexity}
Let $M(t)$ be a solution of the flow $(\ref{flow-eq})$
in $\mathbb{H}^{n+1}$.
If the initial hypersurface $M_0$ in $\mathbb{H}^{n+1}$ satisfies
\begin{equation}
\kappa_i > 1,
\end{equation}
then this condition will also be satisfied by the flow hypersurfaces $M(t)$ during the evolution.
\end{thm}
\begin{proof}
The tensor
\begin{equation}
\label{S_ij-def}
S_{ij} = h_{ij} - g_{ij}
\end{equation}
satisfies the equation
\begin{equation}
\begin{split}
\dot{S}_{ij} - F^{kl} S_{ij;kl} = &\, F^{kl} h_{rk} h^r_l h_{ij} - 2F h_i^k h_{kj}\\
& + K_N \{ 2F g_{ij} - F^{kl} g_{kl} h_{ij}\} + 2F h_{ij} + F^{kl,rs} h_{kl;i} h_{rs;j}\\
\equiv & \, N_{ij} + \tilde{N}_{ij},
\end{split}
\end{equation}
where $\tilde{N}_{ij} = F^{kl,rs} h_{kl;i} h_{rs;j}$.
At every point where $h_{ij} \eta^j = \eta_i$  there holds
\begin{equation}
N_{ij} \eta^i \eta^j = \, \{ F^{kl} h_{rk} h_l^r - 2F + F^{kl} g_{kl} \} \abs{\eta}^2 \geq 0.
\end{equation}
It was proved in \cite[Theorem 3.3, Lemma 4.4]{And07} that
\begin{equation}
\tilde{N}_{ij} \eta^i \eta^j + \sup_{\varGamma} 2 F^{kl} \{ 2 \varGamma_l^r S_{ir;k} \eta^i - \varGamma_k^r \varGamma_l^s S_{rs}\} \geq 0,
\end{equation}
where only the inverse concavity of $F$ was used. Andrews' maximum principle in \cite[ Theorem 3.2]{And07} implies that $S_{ij} > 0$ during
the evolution.
\end{proof}
In the next step we use a constant rank theorem to allow the condition $\kappa_i \geq 1$ in the proof of the succeeding Lemma \ref{pinching-estimates-thm}.
\begin{lem}
Let $M(t)$ be a solution of the flow $(\ref{flow-eq})$ in $\mathbb{H}^{n+1}$ and assume that the tensor $S_{ij}$ satisfies  $S_{ij} \geq 0$
on the hypersurfaces $M(t)$ for $t \in [0, T^*)$, then  $S_{ij}$ is of constant rank $l(t)$ for every $t \in (0,T^*)$.
\end{lem}
\begin{proof}
The proof is similar to those in \cite[Theorem 3.2]{Wan14}, where the main part is based on the computation in \cite[Theorem 3.2]{Bia09}.
For $\epsilon>0$, let
\begin{equation}
W_{ij} = S_{ij} + \epsilon g_{ij}.
\end{equation}
Let $l(t)$ be the minimal rank of $S_{ij}(t)$. For a fixed $t_0 \in (0,T^*)$, let $x_0 \in M(t_0)$ be the point such that $S_{ij}(t_0,\xi)$
attains its minimal rank at $x_0$. Set
\begin{equation}
\phi(t,\xi) = H_{l+1}(W_{ij}(t,\xi)) + \frac{H_{l+2}(W_{ij}(t,\xi))}{H_{l+1}(W_{ij}(t,\xi))},
\end{equation}
where $H_l$ is the elementary symmetric polynomials of eigenvalues of $W_{ij}$, homogeneous of order $l$.
A direct computation shows
\begin{equation}
\begin{split}
F^{kl} W_{ij;kl} - \dot{W}_{ij} =& -F^{kl} h_{rk} h^r_l W_{ij} - F^{kl} g_{kl} W_{ij} + 2F h_i^k W_{kj}\\
&\, - F^{kl,rs} W_{kl;i} W_{rs;j} +2F \epsilon g_{ij} \\
&\,- (1 - \epsilon) \{ F^{kl}h_{rk} h^r_l - 2F + F^{kl} g_{kl}\} g_{ij}.
\end{split}
\end{equation}
As in \cite{Bia09}, we consider a neighborhood  $ (t_0 - \delta ,t_0] \times \mathcal{O}$ around $(t_0,\xi_0)$.
We use the notation $h = O(f)$
if $\abs{h(\xi)} \leq C f(\xi)$ for every $(t,\xi) \in (t_0 - \delta ,t_0] \times \mathcal{O}$, where $C$ is a
constant, depending on the $C^{1,1}$ norm of the second fundamental form on $ (t_0 - \delta ,t_0] \times \mathcal{O}$, but independent of $\epsilon$.
It was proved in \cite[Corollary 2.2]{Bia09} that $\phi$ is in $C^{1,1}$.
And as in \cite{Bia09}, let $G= \{ n-l+1, n-l+2, \dots, n\}$ and $B= \{ 1, \dots, n-l\}$.
We choose the coordinates such that $h_{ij} = \kappa_i \delta_{ij}$ and $g_{ij} = \delta_{ij}$.
In view of \cite[(3.14)]{Bia09}, in such coordinates $\phi^{ij}$ is up to $O(\phi)$ non-negative in $\mathcal{O}$ and we have
\begin{equation}
\begin{split}
F^{kl} \phi_{;kl} - \dot{\phi} \leq& \, \phi^{ij} \{ -F^{kl} h_{rk} h^r_l W_{ij} - F^{kl} g_{kl} W_{ij} + 2F h_i^k W_{kj}\\
& \, +2F \epsilon g_{ij} - F^{kl,rs} W_{kl;i} W_{rs;j}\} + F^{kl} \phi^{ij,rs} W_{ij;k} W_{rs;l}+ O (\phi).
\end{split}
\end{equation}
We can choose $\mathcal{O}$ small enough, such that $\epsilon = O (\phi)$ as in \cite[(3.8)]{Bia09}. It was proved in \cite[(3.14)]{Bia09} that
$\phi^{ii} = O(\phi)$ for $i \in G$ and since $W_{ii} \leq \phi$ for $i \in B$, we infer that
\begin{equation}
F^{kl} \phi_{;kl} - \dot{\phi} \leq - \phi^{ij} F^{kl,rs} W_{kl;i} W_{rs;j} + F^{kl} \phi^{ij,rs} W_{ij;k} W_{rs;l} + O (\phi).
\end{equation}
Using the inverse concavity of $F$ and proceed as in \cite[Theorem 3.2]{Bia09}, we conclude
\begin{equation}
F^{kl} \phi_{;kl} - \dot{\phi} \leq C \{ \phi + \abs{D \phi} \},
\end{equation}
where $C$ is a constant independent of $\epsilon$ and $\phi$.
Taking $\epsilon \rightarrow 0$, the strong maximum principle for parabolic equations yields
\begin{equation}
H_{l(t_0)+1} (S_{ij}(t,\xi)) \equiv 0 \quad \forall (t,\xi) \in (t_0 - \delta, t_0] \times \mathcal{O}.
\end{equation}
Since $M(t_0)$ is a closed hypersurface, $S_{ij}(t_0,\xi)$ is of constant rank $l(t_0)$ on $M(t_0)$.
\end{proof}
Note that the proof of the Lemma \ref{preservation-h-convexity} implies, if the initial hypersurface satisfies $\kappa_i \geq 1$, then this
condition remains true during the evolution. Furthermore, every closed hypersurface in $\mathbb{H}^{n+1}$ contains a point on which holds $\kappa_i>1$.
Thus we conclude
\begin{cor}
Let $M(t)$ be a solution of the flow $(\ref{flow-eq})$
in $\mathbb{H}^{n+1}$. If the initial hypersurface $M_0$ in $\mathbb{H}^{n+1}$ satisfies $\kappa_i \geq 1$, then $\kappa_i>1$ for every
$t \in (0, T^*)$.
\end{cor}
\begin{lem}
\label{pinching-estimates-thm}
Let $M(t)$ be a solution of the flow $(\ref{flow-eq})$
in $\mathbb{H}^{n+1}$ under assumption \ref{general-assmption} $(1)$,
then there exists a uniform positive constant $\epsilon>0$ such that
\begin{equation}
\kappa_1 \geq \epsilon \kappa_n
\end{equation}
during the evolution, where the principal curvatures are labeled as
\begin{equation}
\kappa_1 \leq \cdots \leq \kappa_n.
\end{equation}
\end{lem}
\begin{proof}
The proof is similar to \cite[Lemma 4.2]{Mak12}. By Replacing $M_0$ by $M(t_0)$ for a $t_0 \in (0,T^*)$ as initial hypersurface,
we can assume that $\kappa_i >1$ on $M_0$
Let $F$ be a concave and inverse concave curvature function, then
\begin{equation}
T_{ij} = h_{ij} - g_{ij} - \epsilon (H-n) g_{ij}
\end{equation}
satisfies the equation
\begin{equation}
\begin{split}
\dot{T}_{ij} - F^{kl} T_{ij;kl} = & \,F^{kl} h_{rk} h^r_l \{ h_{ij} - \epsilon H g_{ij}\} - 2 F h_i^k \{ h_{kj} - \epsilon H g_{kj}\}\\
& + 2 K_N F g_{ij} - 2 n \epsilon K_N F g_{ij} - K_N F^{kl} g_{kl} \{ h_{ij} - \epsilon H g_{ij} \}\\
&- 2 F (\epsilon n -1) h_{ij} + F^{kl,rs} h_{kl;i} h_{rs;j} - \epsilon F^{kl,rs} h_{kl;p} h_{rs;q} g^{pq} g_{ij}\\
\equiv & \, N_{ij} + \tilde{N}_{ij},
\end{split}
\end{equation}
where $\tilde{N}_{ij} = F^{kl,rs} h_{kl;i} h_{rs;j} - \epsilon F^{kl,rs} h_{kl;p} h_{rs;q} g^{pq} g_{ij}$.\\
At every point where $T_{ij} \eta^j =0$ there holds
\begin{equation}
\begin{split}
N_{ij} \eta^i \eta^j = & \, F^{kl} h_{rk} h^r_l (1 - \epsilon n) \abs{\eta}^2 + 2F h_{ij} (\epsilon n -1) \eta^i \eta^j \\
& + \{ F^{kl} g_{kl} - 2F \} (1 - \epsilon n ) \abs{\eta}^2 -2 F (\epsilon n -1) h_{ij} \eta^i \eta^j\\
= & \, (1 - \epsilon n ) \sum_{i} F_i( \kappa_i^2 - 2 \kappa_i + 1) \abs{\eta}^2 \geq 0.
\end{split}
\end{equation}
It is proved in \cite[Theorem 4.1]{And94} (see also the modification in \cite[Theorem B.2]{Mak12}) that
\begin{equation}
\tilde{N}_{ij} \eta^i \eta^j + \sup_{\varGamma} 2 F^{kl} \{ 2 \varGamma_l^r T_{ir;k} \eta^i - \varGamma_k^r \varGamma_l^s T_{rs}\} \geq 0,
\end{equation}
We can choose $\epsilon >0$ sufficiently small, such that $T_{ij} \geq 0$ on $M_0$, then the
Andrews' maximum principle \cite[Theorem 3.2]{And07} implies $T_{ij} \geq 0$ and hence
\begin{equation}
\kappa_1 -1 \geq \epsilon  (H-n)
\end{equation}
during the evolution.
\end{proof}
The following pinching results is due to Gerhardt. By using \cite[Theorem 1.1]{Ger15} and the duality result Lemma \ref{duality-result} we obtain
\begin{thm}
Let $M(t)$ be a solution of the flow $(\ref{flow-eq})$
in $\mathbb{H}^{n+1}$ under the assumption \ref{general-assmption} $(2)$,
then there exists a uniform constant $\epsilon>0$ such that
\begin{equation}
\kappa_1 \geq \epsilon \kappa_n
\end{equation}
during the evolution.
\end{thm}

\section{Contracting flows - convergence to a point}
Fix a point $p_0 \in \mathbb{H}^{n+1}$, the hyperbolic metric in the geodesic polar coordinates centered at $p_0$ can be expressed as
\begin{equation}
\ud \bar{s}^2 = dr^2 + \sinh^2 r \sigma_{ij} dx^i dx^j,
\end{equation}
where $\sigma_{ij}$ is the canonical metric of $\mathbb{S}^n$.\\
Geodesic spheres with center in $p_0$ are totally umbilic. The induced metric, second fundamental form and the principal curvatures of the
coordinate slices $S_r = \{ x^0 = r\}$ are given by
\begin{equation}
\bar{g}_{ij} = \sinh^2 r \sigma_{ij}, \, \bar{h}_{ij} = \tfrac{1}{2} \dot{\bar{g}}_{ij} = \coth r \bar{g}_{ij}, \, \bar{\kappa}_i = \coth r,
\end{equation}
respectively. See \cite[(1.5.12)]{Ger06}.
\begin{lem}
\label{spherical-flow}
Consider $(\ref{flow-eq})$ with initial hypersurface $x(0) = S_{r_0}$, then the corresponding flow exists in a maximal time intervall
$[0, T^*)$ with $T^*$ finite and will shrink to a point. The flow hypersurfaces $M(t)$ are all geodesic spheres with the same center and
their radii $\Theta= \Theta(t)$ solve the ODE
\begin{equation}
\begin{split}
\label{ode-equation}
&\dot{\Theta} = -\coth \Theta,\\
&\Theta(0) = r_0.
\end{split}
\end{equation}
\end{lem}
\begin{proof}
We set
\begin{equation}
\begin{split}
\label{ode-flow}
x^0(t,\xi) &= \Theta(t),\\
x^i(t, \xi) &= x^i(0 , \xi).
\end{split}
\end{equation}
In view of \cite[(1.5.7)]{Ger06} the exterior normal of a geodesic sphere is (1,0,\dots,0). Using that $F(\bar{h}^i_j) = \coth \Theta$,
we see that $x$ in (\ref{ode-flow}) solves the flow equation (\ref{flow-eq}). Now the solution of (\ref{ode-equation}) is given by
\begin{equation}
\label{spherical-sol}
\cosh \Theta = (\cosh r_0) e^{-t}.
\end{equation}
Thus the spherical flow exists only for a finite time $[0,T^*)$. Note that (\ref{spherical-sol}) can be rewritten as
\begin{equation}
\label{sol-spherical-flow}
\Theta = \mathrm{arccosh} \, e^{(T^*-t)}.
\end{equation}
\end{proof}

Next we want to prove that the flow (\ref{flow-eq}) shrinks to a point.
Using the inverse of the Beltrami map, $\mathbb{H}^{n+1}$ is parametrizable over $B_1(0)$ yielding the metric (cf. \cite[Section 10.2]{Ger06})
\begin{equation}
d \bar{s}^2 =  \frac{1}{(1-r^2)^2} \, dr^2 + \frac{r^2}{1-r^2}\sigma_{ij} d \xi^i d \xi^j.
\end{equation}

Define the variable $\varrho$ by
\begin{equation}
\label{relation-rho-r}
\varrho = \mathrm{arctanh} r = \tfrac{1}{2} (\log(1+r) -\log(1-r)),
\end{equation}
then
\begin{equation}
d \bar{s}^2  = d \varrho^2 + \sinh^2 \varrho \, \sigma_{ij}\, d \xi^i d \xi^j.
\end{equation}
Let
\begin{equation}
\label{euclidean-polar-coordinates}
d \tilde{s}^2 = dr^2 + r^2 \sigma_{ij} d \xi^i d \xi^j
\end{equation}
be the Euclidean metric over $B_1(0)$.
Define
\begin{equation}
d \tau = \frac{1}{r \sqrt{1-r^2}} dr, \quad d \tilde{\tau} = r^{-2}  dr,
\end{equation}
we have further
\begin{equation}
\begin{split}
d \bar{s}^2 &= \frac{r^2}{1-r^2} \{ d \tau^2 + \sigma_{ij} d \xi^i d \xi^j \} \equiv e^{2 \psi} \{ d \tau^2 + \sigma_{ij} d \xi^i d \xi^j \},\\
d \tilde{s}^2 &= r^2 \{ d \tilde{\tau}^2 + \sigma_{ij} d \xi^i d \xi^j \} \equiv e^{2 \tilde{\psi}} \{ d \tilde{\tau}^2 + \sigma_{ij} d \xi^i d \xi^j\}.
\end{split}
\end{equation}
An arbitary closed, connected, strictly embedded hypersurface $M \subset \mathbb{H}^{n+1}$ bounds a convex body and we can write $M$
as a graph in geodesic polar coordinates.
\begin{equation}
M = \mathrm{graph }\,  u = \{ \tau = u(x): x \in \mathbb{S}^n\}.
\end{equation}
$M$ can also be viewed as a graph $\tilde{M}$ in $B_1(0)$ with respect to the Euclidean metric
\begin{equation}
\tilde{M} = \mathrm{graph } \, \tilde{u} = \{\ \tilde{\tau} = \tilde{u} (x): x \in \mathbb{S}^n \}.
\end{equation}
Writing $\tilde{u} = \varphi (u)$, then there holds (see \cite[(10.2.18)]{Ger06})
\begin{equation}
\dot{\varphi}^2 = 1 - r^2.
\end{equation}

The same argument as in \cite[Lemma 6.1]{Ger13} yields
\begin{lem}
\label{compare-spherical-barrier}
Let $M(t)$ be a solution of $(\ref{flow-eq})$ on a maximal time interval $[0,T^*)$ and represent $M(t)$, for a fixed $t \in [0,T^*)$, as
a graph in polar coordinates with center in $x_0 \in \hat{M}(t)$
\begin{equation}
M(t) = \mathrm{graph } \, u(t, \cdot),
\end{equation}
then
\begin{equation}
\inf_{M(t)} u \leq \Theta(t,T^*) \leq \sup_{M(t)} u,
\end{equation}
where the solution of the spherical flow $\Theta(t,T^*)$ is given by $(\ref{sol-spherical-flow})$. \qed
\end{lem}

\begin{lem}
\label{boundedness-tilde-u}
Let $x_0 \in \hat{M}(t)$ be as above and represent $M(t)$ in Euclidean polar coordinates $(\ref{euclidean-polar-coordinates})$,
 then there exists a constant $c_0 = c_0(M_0)<1$ such that the estimate
\begin{equation}
r \leq c_0
\end{equation}
holds for any $t \in [0, T^*).$
\end{lem}
\begin{proof}
The argument is similar to those in \cite[Lemma 6.3, Remark 6.5]{Ger13}. Looking at the scalar flow equation for a short time
interval, we conclude that the convex bodies $\hat{M}(t) \subset \mathbb{H}^{n+1}$ are decreasing with respect to $t$. Furthermore, $\hat{M}_0$ is
strictly convex. Thus $ \varrho $ is uniformly bounded and the claim follows from the relation
\begin{equation}
r = \tanh \varrho = 1 - \frac{2}{e^{2 \varrho} +1}.
\end{equation}
\end{proof}

Denote $h_{ij}$ resp. $\tilde{h}_{ij}$ the second fundamental forms and $\kappa_i$ resp $\tilde{\kappa}_i$ the principal curvatures
of $M$ with respect to the ambient metric $\bar{g}_{\alpha \beta}$ resp.
$\tilde{g}_{\alpha \beta}$.
\begin{lem}
The principal curvatures $\tilde{\kappa}_i$ of $M(t)$ are pinched, i.e., there exists a uniform constant $c$ such that
\begin{equation}
\label{pinching-estimates-euclidean}
\tilde{\kappa}_n \leq c \tilde{\kappa}_1,
\end{equation}
where the $\tilde{\kappa}_i$ are labeled as
\begin{equation}
\tilde{\kappa}_1 \leq \cdots \leq \tilde{\kappa}_n.
\end{equation}
\end{lem}
\begin{proof}
The $h_{ij}$ and $\tilde{h}_{ij}$ are related through the formula (see \cite[(10.2.33)]{Ger06})
\begin{equation}
\tilde{h}_{ij} \tilde{v} = (1-r^2) h_{ij} v,
\end{equation}
where
\begin{equation}
\begin{split}
v^2 = 1 + \sigma^{ij} u_i u_j,\\
\tilde{v}^2 = 1 +  \dot{\varphi}^2 \sigma^{ij} u_i u_j.
\end{split}
\end{equation}
Because of Lemma \ref{boundedness-tilde-u} there exists $0 < \delta <1$ such that
\begin{equation}
r^2 \leq 1 - \delta,
\end{equation}
and thus
\begin{equation}
\delta v^2 \leq \tilde{v}^2 \leq v^2,
\end{equation}
\begin{equation}
\delta h_{ij} \leq \tilde{h}_{ij} \leq \delta^{-1} h_{ij}.
\end{equation}
Furthermore, there holds
\begin{equation}
\begin{split}
g_{ij} &= \frac{r^2}{1-r^2} \{ u_i u_j + \sigma_{ij}\},\\
\tilde{g}_{ij} &= r^2 \{ \dot{\varphi}^2 u_i u_j + \sigma_{ij}\}.
\end{split}
\end{equation}
and we conclude
\begin{equation}
\delta^2 g_{ij} \leq \tilde{g}_{ij} \leq g_{ij}.
\end{equation}
Now the claim follows from the maximum-minimum principle.
\end{proof}

For $\hat{M}(t) \subset \mathbb{H}^{n+1}$, the inradius $\rho_-(t)$ and circumradius $\rho_+(t)$ of $\hat{M}(t)$ are defined by
\begin{equation}
\begin{split}
\rho_-(t) =& \sup \{ r: B_r(y) \textrm{ is enclosed by } \hat{M}(t) \textrm{ for some } y \in \mathbb{H}^{n+1}\},\\
\rho_+(t) =& \inf \{ r: B_r(y) \textrm{ encloses } \hat{M}(t) \textrm{ for some } y \in \mathbb{H}^{n+1}\}.
\end{split}
\end{equation}
Now, choose $x_0 \in \hat{M}(t)$ to be the center of the inball of $\hat{M}(t) \subset \mathbb{H}^{n+1}$ and let $x_0$ be
the center of the geodesic polar coordinates. Note that the center of the Euclidean inball is also $x_0$.
Let $\rho_-(t)$ resp. $\rho_+(t)$ be the inradius resp. circumradius of $\hat{M}(t) \subset \mathbb{H}^{n+1}$,
and let $\tilde{\rho}_-(t)$ resp. $\tilde{\rho}_+(t)$ be the inradius resp.
circumradius of $\hat{M}(t) \subset \mathbb{R}^{n+1}$.
\begin{lem}
\label{shrink-the-convex-body}
Let $B_{{\rho}_-(t)}(x_0) \subset \hat{M}(t)$ be a geodesic inball, then there exist positive constants $c$ and $\delta$, such that
\begin{equation}
\hat{M}(t) \subset B_{4c \rho_-(t)}(x_0) \quad \forall t \in [T^* - \delta, T^*).
\end{equation}
\end{lem}
\begin{proof}
The pinching estimates in the Euclidean ambient space (\ref{pinching-estimates-euclidean}) and
\cite[Theorem 5.1, Theorem 5.4]{And94}  imply
\begin{equation}
\label{pinching-ineq}
\tilde{\rho}_+(t) \leq c \tilde{\rho}_-(t)
\end{equation}
with a uniform constant $c$, hence $\hat{M}(t)$ is contained in the Euclidean ball $B_{\tilde{\rho}}(0)$,
\begin{equation}
\hat{M}(t) \subset B_{\tilde{\rho}}(0), \quad \tilde{\rho}(t) = 2c \tilde{\rho}_-(t).
\end{equation}
Furthermore, we deduce from Lemma \ref{compare-spherical-barrier} that
\begin{equation}
\label{euclidean-ineq}
\inf_{M(t)} \tilde{u} \leq \tilde{\Theta} \leq \sup_{M(t)} \tilde{u},
\end{equation}
where $M(t) = \textrm{graph }\, \tilde{u}$ is a representation of $M(t)$ in Euclidean polar coordinates.
We conclude further
\begin{equation}
\tilde{\rho}(t) = 2c \tilde{\rho}_-(t) \leq 2c \tilde{\Theta}.
\end{equation}
Choose now $\delta>0$ small such that
\begin{equation}
2c \tilde{\Theta}(t,T^*) \leq 1 \quad \forall t \in [T^*-\delta,T^*).
\end{equation}
Now it holds for
\begin{equation}
\rho(t) = \mathrm{ arctanh } \, \tilde{\rho}(t)
\end{equation}
\begin{equation}
\hat{M}(t) \subset B_{\rho(t)}(x_0) \subset \mathbb{H}^{n+1}.
\end{equation}
Since
\begin{equation}
\tilde{\rho}(t) \leq 1,
\end{equation}
we conclude further
\begin{equation}
\tilde{\rho} \leq \rho \leq 2 \tilde{\rho}, \quad \tilde{\rho}_- \leq \rho_-.
\end{equation}
Thus
\begin{equation}
\rho \leq 2 \tilde{\rho} = 4c \tilde{\rho}_- \leq 4c \rho_-
\end{equation}
and the claim follows.
\end{proof}
\begin{lem}
\label{smooth-estimates-u}
During the evolution the flow hypersurfaces $M(t)$ are smooth and uniformly convex satisfying a priori estimates in any compact subinterval
$[0,T] \subset [0, T^*)$.
\end{lem}
\begin{proof}
Let $0<T < T^*$ be fixed.
From (\ref{pinching-ineq}) and (\ref{euclidean-ineq}) we infer
\begin{equation}
c \tilde{\Theta}(T, T^*) \leq  \tilde{\rho}_- (T).
\end{equation}
Since
\begin{equation}
\Theta(T,T^*) = \mathrm{arctanh} \tilde{\Theta}(T,T^*), \quad \rho_-(T) = \mathrm{arctanh} \tilde{\rho}_-(T),
\end{equation}
and $\tilde{\rho}_-(T)$, $\tilde{\Theta}(T,T^*)$ are uniformly bounded from above by 1 we infer that
\begin{equation}
0< \tfrac{c}{2} \Theta = \tfrac{c}{2} \mathrm{arctanh} \tilde{\Theta} \leq c \tilde{\Theta} \leq \mathrm{arctanh} (c \tilde{\Theta}) \leq \rho_-(T).
\end{equation}
Let $x_0 \in \hat{M}(T)$ be the center of an inball and introduce geodesic polar coordinates with center $x_0$.
This coordinate system will cover the flow in $0 \leq t \leq T$. Writing the flow hypersurfaces as graphs $u(t,\cdot)$ of a function we have
\begin{equation}
0 < c^{-1} \leq u \leq c.
\end{equation}
And since $M(t)$ are convex,
\begin{equation}
v^2 = 1 + \sinh^{-2} u \, \sigma^{ij} u_i u_j
\end{equation}
is uniformly bounded.
Under assumption \ref{general-assmption} $(1)$ we have $\kappa_i \geq 1$. And under assumption \ref{general-assmption} $(2)$ it is proved in
\cite[Lemma 4.4]{Ger15} that
\begin{equation}
\frac{1}{n} \tilde{\kappa}_n \leq \tilde{F} \leq c
\end{equation} in $N$ or equivalently, $\kappa_i \geq c$ in $\mathbb{H}^{n+1}$.
The proof of uniform boundedness of $\kappa_i$ from above is similar to those in \cite[Theorem 6.6]{Ger13}.
Since $F$ is concave, we may first apply the Krylov-Safonov and then the parabolic Schauder estimates to obtain the desired a priori estimates.
\end{proof}

In view of Lemma \ref{spherical-flow}, \ref{compare-spherical-barrier}, \ref{shrink-the-convex-body} and \ref{smooth-estimates-u},
the flow (\ref{flow-eq}) shrinks in finite time to a point $x_0$.

\section{The rescaled flow}
In view of Lemma \ref{compare-spherical-barrier} and \ref{shrink-the-convex-body} we can choose $\delta>0$ small and define
\begin{equation}
t_\delta = T^* - \delta,
\end{equation}
such that
\begin{equation}
\hat{M} (t_\delta) \subset B_{8c \rho_- (t_\delta)}(x_0) \quad \forall x_0 \in \hat{M}(t_\delta),
\end{equation}
and
\begin{equation}
8c \rho_- (t_\delta) \leq 8c \Theta(t_\delta, T^*) < 1.
\end{equation}

Fix now a $t_0 \in (t_\delta, T^*)$ and let $B_{\rho_- (t_0)} (x_0)$ be an inball of $\hat{M}(t_0)$. Choose $x_0$ to be the center of a
geodesic polar coordinate system, then the hypersurfaces $M(t)$ can be written as graphs
\begin{equation}
M(t) = \mathrm{graph \,}u(t, \cdot) \quad \forall t_\delta \leq t \leq t_0,
\end{equation}
such that
\begin{equation}
\rho_- (t_0) \leq u(t_0) \leq u(t) \leq 1.
\end{equation}
\begin{lem}
\label{chi-i-u-i}
Let
\begin{equation}
\chi = \frac{v}{\sinh u} \equiv v \eta(r),
\end{equation}
if $\chi_i =0$, then $u_i=0$.
\end{lem}
\begin{proof}
Note that
\begin{equation}
\eta(r) = \frac{1}{\sinh r}
\end{equation}
solves the equation
\begin{equation}
\dot{\eta} = - \frac{\bar{H}}{n} \eta,
\end{equation}
hence the proof is same as those in \cite[Lemma 7.1]{Ger13}.
\end{proof}
Similar to \cite[Lemma 7.2, Corollary 7.3]{Ger13} we obtain
\begin{lem}
\label{rescaled-curvature-bounded-from-above}
There exists a uniform constant $c > 0$ such that
\begin{equation}
\Theta(t,T^*) F \leq c \quad \forall t \in [t_\delta,T^*),
\end{equation}
and that the rescaled principal curvatures $\tilde\kappa_i = \kappa_i \Theta$ satisfy
\begin{equation}
\label{upper-bound-rescaled-principal-curvature}
\tilde\kappa_i \leq c \quad \forall t \in [t_\delta,T^*).
\end{equation}
\qed
\end{lem}
\begin{lem}
\label{estimates-varphi}
Let $t_1 \in [t_\delta, T^*)$ be arbitrary and let $t_2> t_1$ be such that
\begin{equation}
\Theta (t_2,T^*) = \tfrac{1}{2} \Theta(t_1, T^*).
\end{equation}
Let $x_0 \in \hat{M}(t_2)$ be the center of an geodesic inball and introduce polar coordinates around $x_0$ and write the hypersurface $M(t)$ as graphs
\begin{equation}
M(t) = \mathrm{graph}\, u(t, \cdot).
\end{equation}
Define $\vartheta$ by
\begin{equation}
\vartheta(r) = \sinh r,
\end{equation}
and
\begin{equation}
\varphi = \int_{r_2}^u \vartheta^{-1},
\end{equation}
where $r_2  = \Theta(t_2, T^*)$. Then $\varphi(t, \cdot)$ is uniformly bounded in $C^2( \mathbb{S}^n)$ for any $t_1 \leq t \leq t_2$ independent
of $t_1,t_2$. Furthermore, let $\varGamma_{ij}^k$ and $\tilde\varGamma_{ij}^k$ be the Christoffel symbols of the metrics $g_{ij}$ and $\sigma_{ij}$
respectively, then the tensor $\varGamma_{ij}^k - \tilde\varGamma_{ij}^k$ is also uniformly bounded independent of $t_1, t_2$.
\end{lem}
\begin{proof}
As in \cite[Lemma 7.4]{Ger13}, we conclude from Lemma \ref{compare-spherical-barrier} and Lemma \ref{shrink-the-convex-body}
that there exists a uniform constant $c>1$, independent of $t_1, t_2$, such that
\begin{equation}
\label{C0-estimates-u}
c^{-1} \Theta(t_2, T^*) \leq u(t,\xi) \leq c \Theta(t_2, T^*) \quad \forall t \in [t_1,t_2].
\end{equation}
Note that
\begin{equation}
\varphi= \{ \log \sinh (\tfrac{r}{2}) - \log \cosh (\tfrac{r}{2})\} \big|_{r_2}^u,
\end{equation}
thus we derive the $C^0$-estimates
\begin{equation}
\abs{\varphi} \leq \log c.
\end{equation}
As in the proof of \cite[Lemma 7.5]{Ger13}, an upper bound for the principal curvatures
of the slices $\{ x^0 = \mathrm{const}\}$ intersecting $M(t)$ satisfies
\begin{equation}
\bar{\kappa} \leq \frac{\sup \cosh u(0, \cdot)}{\sinh u_{\min}} \leq \frac{c}{u_{\min}},
\end{equation}
and from \cite[(2.7.83)]{Ger06} we infer that the uniformly boundedness of $v$.
\begin{equation}
v \leq e^{\bar{\kappa} (u_{\max} - u_{\min})} \leq e^{ c ( \frac{u_{\max}}{u_{\min}} ) - 1},
\end{equation}
concluding further that
\begin{equation}
\abs{D \varphi}^2 = v^2 -1 \leq c.
\end{equation}
Define
\begin{equation}
\tilde{g}^{ij} = \sigma^{ij} - v^{-2} \varphi^i \varphi^j,
\end{equation}
where
\begin{equation}
\varphi^i = \sigma^{ik} \varphi_k.
\end{equation}
Due to the boundedness of $v$ the metrics $\tilde{g}_{ij}$ and $\sigma_{ij}$ are equivalent, thus we can raise the indices of
$\varphi_{ij}$ by $\tilde{g}_{ij}$ and by employing
the relation \cite[(3.26)]{Ger11}
\begin{equation}
h^i_j = v^{-1} \vartheta^{-1} \{ - (\sigma^{ik} - v^{-2} \varphi^i \varphi^k) \varphi_{jk} + \dot{\vartheta} \delta^i_j \},
\end{equation}
we infer
\begin{equation}
\tilde{g}^{ik} \varphi_{jk} = -v \vartheta h^i_j + \dot{\vartheta} \delta^i_j,
\end{equation}
concluding further from (\ref{upper-bound-rescaled-principal-curvature})
\begin{equation}
\norm{ \varphi_{ij} }^2 \leq c (v^2 \vartheta^2 \abs{A}^2 + n \dot{\vartheta}^2)
\end{equation}
is bounded from above for all $t \in [t_1, t_2]$.
We choose coordinates such that $\tilde{\varGamma}_{ij}^k$ in a fixed point vanishes.
Denote the covariant derivative with respect to $\sigma_{ij}$ by a colon. In such coordinates
\begin{equation}
\varGamma_{ij}^k = \tfrac{1}{2} g^{km} (g_{mi:j} + g_{mj:i} - g_{ij:m}).
\end{equation}
From
\begin{equation}
\label{tilde-g_ij}
g^{ij} = \vartheta^{-2} \tilde{g}^{ij}
\end{equation}
we compute
\begin{equation}
g^{km} g_{mi:j} =\tilde{g}^{km} \{ \varphi_{mj} \varphi_i + \varphi_{ij} \varphi_m +
2 \cosh u \, \varphi_j (\varphi_m \varphi_i + \sigma_{mi}) \}.
\end{equation}
Using the estimates for $\varphi$ proved before, we conclude that
$\varGamma_{ij}^k - \tilde{\varGamma}_{ij}^k$ are uniformly bounded independent of $t_1$ and $t_2$.
\end{proof}

Define a new time parameter as
\begin{equation}
\tau = - \log \Theta,
\end{equation}
then
\begin{equation}
\frac{dt}{d \tau} = \Theta \frac{\sinh \Theta}{\cosh \Theta}.
\end{equation}
In the following we denote the differentiation with respect to $t$ by a dot and differentiation with respect to $\tau $ by a prime.
\begin{lem}
\label{estimates-tilde-F}
The rescaled quantity $\tilde{F} = F \Theta$ satisfies the inequality
\begin{equation}
\sup_{M(t_1)} \tilde{F} \leq c \inf_{M(t_2)} \tilde{F}
\end{equation}
with a uniform constant $c > 0$.
\end{lem}
\begin{proof}
$\tilde{F}$ satisfies the equation
\begin{equation}
\tilde{F}' = \dot{F} \Theta^2 \frac{\sinh \Theta}{\cosh \Theta} - \tilde{F},
\end{equation}
and from the evolution equation of $F$ in \cite[(2.8)]{Ger13} we conclude further
\begin{equation}
\begin{split}
\label{eq-F-tilde}
\tilde{F}'  + \tilde{F} - \{ F^{ij} F_{;ij}
+ F^{ij} h_{ik} h^k_j  F + K_N F^{ij} g_{ij} F \}\Theta^2 \frac{\sinh \Theta}{\cosh \Theta}=0.
\end{split}
\end{equation}
We consider the non-trivial term in (\ref{eq-F-tilde})
\begin{equation}
- F^{ij} F_{;ij} \Theta^2 \frac{\sinh \Theta}{\cosh \Theta}.
\end{equation}
In view of (\ref{tilde-g_ij}), the pinching estimate and the boundedness of  $v$,
$\Theta^2 F^{ij}$ and $\sigma^{ij}$ are equivalent and hence uniformly positive definite.
Furthermore,
\begin{equation}
F_{;ij} = F_{:ij} - \{ \varGamma_{ij}^k - \tilde{\varGamma}_{ij}^k \} F_k.
\end{equation}
Hence we conclude from Lemma \ref{estimates-varphi} that $\tilde{F}$ satisfies a uniform parabolic equation of the form
\begin{equation}
\tilde{F}' - a^{ij} \tilde{F}_{:ij} + b^i \tilde{F}_i + c \tilde{F} =0
\end{equation}
in the cylinder $[\tau_1, \tau_2] \times \mathbb{S}^n$, where $\tau_i = - \log \Theta(t_i, T^*)$, with uniformly bounded coefficients.
The statement follows then from the parabolic Harnack inequality.
\end{proof}
\begin{cor}
\label{rescaled-curvature-bounded-from-below}
The rescaled principal curvatures $\tilde{\kappa}_i = \kappa \Theta$ are uniformly bounded from below.
\end{cor}
\begin{proof}
Consider a point $(t, \xi)$ in $M(t)$ such that
\begin{equation}
u(t,\xi) = \sup_ {M(t)} u.
\end{equation}
In view of \cite[(1.5.10)]{Ger06}, it holds in $(t, \xi)$
\begin{equation}
h_{ij} \geq \bar{h}_{ij}, \quad g_{ij} = \bar{g}_{ij}, \quad \kappa_i \geq \bar{\kappa}_i = \frac{\cosh u}{\sinh u},
\end{equation}
where we denote the quantity of the slices $\{ x^0 = \mathrm{const} \, \}$ with a bar.
In view of (\ref{C0-estimates-u})
\begin{equation}
\sup_{M(t)} \tilde{F} \geq F(\tilde{\kappa}_i (t,\xi)) \geq F \left( \frac{\cosh u(t,\xi)}{\sinh u(t, \xi)} \Theta(t, T^*) \right) \geq c >0.
\end{equation}
The statement follows from the pinching estimates and Lemma \ref{estimates-tilde-F}.
\end{proof}
Let $x_0 \in \mathbb{H}^{n+1}$ be the point the flow hypersurfaces are shrinking to and introduce geodesic polar coordinates around it. Write
$M(t) = \textrm{graph } u(t, \cdot)$
and let
\begin{equation}
\tilde{u} (\tau, \xi) = u(t, \xi) \Theta(t, T^*)^{-1},
\end{equation}
\begin{equation}
\tau_\delta = - \log \Theta(t_\delta, T^*), \quad Q(\tau_\delta, \infty) = [\tau_\delta, \infty) \times \mathbb{S}^n.
\end{equation}
Using the same argument as in \cite[Lemma 7.9, Lemma 7.10]{Ger13} we conclude that
\begin{lem}
The quantities $v$ and $|D \tilde{u}|$ are uniformly bounded form above and
$\tilde{u}$ is uniformly bounded from below and above in $Q( \tau_\delta, \infty) $. \qed
\end{lem}
Let
\begin{equation}
\varphi = - \int_u^{\Theta(0,T^*)} \vartheta^{-1},
\end{equation}
then
\begin{equation}
\varphi_i = \vartheta^{-1} u_i, \quad \varphi_{ij} = \vartheta^{-1} u_{ij} - \cosh u\, \vartheta^{-2} \, u_i \, u_j,
\end{equation}
and
\begin{equation}
\vartheta^{-2} |D^2 u|^2 + |D \tilde{u}|^4 \cosh^2 u - 2 \vartheta^{-1} |D^2 u| |D \tilde{u}|^2 \cosh u \leq |D^2 \varphi|^2.
\end{equation}
Since $|D^2 \varphi|$ and $|D \tilde{u}|$ are bounded, we conclude that the $C^2$-norm of $\tilde{u}$ is uniformly bounded, where
the covariant derivatives of $\tilde{u}$ and $\varphi$ are taken with respect to $\sigma_{ij}$.
From \cite[Remark 1.5.1, Lemma 2.7.6]{Ger06} we conclude that
\begin{equation}
\frac{\sinh \Theta}{\cosh \Theta} F v
= \varPhi (x, \tau , \tilde{u}, \tilde{u} e^{-\tau}, D\tilde{u}, D^2 \tilde{u}),
\end{equation}
where $\varPhi$ is a smooth function with respect to its arguments, and
\begin{equation}
\begin{split}
\varPhi^{ij} & \equiv \frac{\partial \varPhi}{\partial (-\tilde{u}_{ij})}
= F^{ij}  \Theta \frac{\sinh \Theta}{ \cosh \Theta},\\
\varPhi^{ij,kl} &= F^{ij,kl}  v^{-1} \Theta^2 \frac{ \sinh \Theta}{ \cosh \Theta}.
\end{split}
\end{equation}
Hence by applying first the Krylov and Safonov, then the Schauder estimates, we deduce (cf. \cite[Remark 2.6.2]{Ger06})
\begin{thm}
\label{regularity-rescaled-function}
The rescaled function $\tilde{u}$ satisfies the uniformly parabolic equation
\begin{equation}
\tilde{u}' = - \varPhi + \tilde{u}
\end{equation}
in $Q(\tau_\delta, \infty)$ and $\tilde{u}(\tau, \cdot)$ satisfies a priori estimates in $C^{\infty} (\mathbb{S}^n)$ independently of $\tau$.
\end{thm}
\section{Convergence to a sphere}
The aim of this section is to prove that $\tilde{u}$ converges exponentially fast to the constant function $1$ if $F$ is strictly concave or
$F = \frac{1}{n} H$. Comparing the proof in \cite[Section 8]{Ger13}, we should handle a term stemming from the negative
curvature of the ambient space $K_N <0$.
\begin{lem}
\label{lem-K_N-term}
There exists a positive constant $C$ such that
\begin{equation}
F^{kl} g_{kl} \abs{A}^2 -FH \leq C \sum_{i<j} (\kappa_i - \kappa_j)^2.
\end{equation}
\end{lem}
\begin{proof}
The proof is similar to \cite[Lemma 8.2]{Ger13}. Let
\begin{equation}
\varphi = F^{kl} g_{kl} \abs{A}^2 -FH. 
\end{equation}
Denote the partial derivatives of $\varphi$ with respect to $\kappa_i$ by $\varphi_i$, then
\begin{equation}
\varphi_j= \sum_{i=1}^n F_{ij} \abs{A}^2 + \sum_{i=1}^n 2 F_i \kappa_j - F_j H - F,\\
\end{equation}
\begin{equation}
\begin{split}
\varphi_{jk} = &\sum_{i=1}^n F_{ijk} \abs{A}^2 + \sum_{i=1}^n 2 F_{ij} \kappa_k + \sum_{i=1}^n 2 F_{ik} \kappa_j \\
 & +2 \delta_{jk} \sum_{i=1}^n F_i - F_{jk} H - F_j - F_k.
\end{split}
\end{equation}
Therefore
\begin{equation}
\varphi(\kappa_n, \cdots, \kappa_n) = 0, \quad \varphi_j(\kappa_n, \cdots, \kappa_n) = 0 \quad \forall j = 1, \dots n.
\end{equation}
by using the Euler's homogeneous relation and the normalization (\ref{normalization-condition}).
Furthermore, $\varphi_{jk}$ are uniformly bounded from above, since $\varphi_{jk}$ are homogeneous of grad 0 and $\frac{\kappa_i}{\abs{A}}$ are
compactly contained in the defining cone.
The statement follows by an argument using Taylor's expansion up to the second order similar to those in \cite[Lemma 8.2]{Ger13}.
\end{proof}
We want to estimate the function
\begin{equation}
f_\sigma = F^{-\alpha} (\abs{A}^2 - nF^2),
\end{equation}
where
\begin{equation}
\alpha = 2 - \sigma,
\end{equation}
and $0< \sigma <1$ small. For simplicity we drop the subscript $\sigma$ of $f_\sigma$.
In the following we always assume that $F$ satisfies the assumption \ref{general-assmption}.

By Lemma \ref{lem-K_N-term} we have the following inequality corresponding to \cite[Lemma 8.3]{Ger13}.
\begin{lem}
Let $F$ be strictly concave, then there exist uniform constants
$\epsilon>0$ and $C>0$, such that
\begin{equation}
\begin{split}
\label{ineq-f}
-F^{ij} f_{ij} + &2\epsilon^2 F^{ij} h_{ki} h^k_j f \leq \alpha F^{-1} F^{ij} F_{;ij} f + 2(\alpha-1) F^{-1} F^{ij} F_i f_j\\
& - 2 \{ h^{ij} - F n F^{ij} \} F^{-\alpha} F_{;ij} - 2 \epsilon^2 \abs{DA}^2 F^{-\alpha} + 2C f.
\end{split}
\end{equation}
\end{lem}
Corresponding to  \cite[Lemma 8.5]{Ger13} we have
\begin{lem}
Let $F$ be strictly concave, then there exist positive constants $C$ and $c$ such that for any $p \geq 2$, any $\delta>0$
and any $0 \leq t< T^*$
\begin{equation}
\begin{split}
\label{ellip-int-ineq}
\epsilon^2 \int_M F^{ij} &h_{ki} h^k_j f^p \leq \{ \delta^{-1} c (p-1) +c \} \int_M F^{ij} f_i f_j f^{p-2}\\
& + \{ \delta c (p-1) +c \} \int_M \abs{DA}^2 F^{-\alpha} f^{p-1} + 2C \int_M f^p.
\end{split}
\end{equation}
\end{lem}
Parallel to \cite[Lemma 8.6]{Ger13} we have
\begin{lem}
Let $F$ be strictly concave, then there exist $C_1>0$ and $\sigma_0>0$ such that for all
\begin{equation}
p \geq 4c \epsilon^{-2}, \quad \sigma \leq \min (\tfrac{1}{4} c^{-1} \epsilon^3 p ^{-1/2}, \sigma_0),
\end{equation}
the estimate
\begin{equation}
\norm{f}_{p,M} \leq C_1 \quad \forall t \in [0, T^*)
\end{equation}
holds, where $C_1 = C_1( M_0, p)$ and $\sigma_0 = \sigma_0(F, M_0)$.
\end{lem}
\begin{proof}
Multiply \cite[(8.30)]{Ger13} with $p f^{p-1}$ and integrate by parts,
and note that
\begin{equation}
d \mu_t = \mu_t \, dx \, \textrm{ on }\, M_t,
\end{equation}
where
\begin{equation}
\frac{d}{dt} \mu_t = \frac{d}{dt} \sqrt{g} = \tfrac{1}{2} \mu_t g^{ij} \dot{g}_{ij} = - FH \mu_t,
\end{equation}
thus
\begin{equation}
\frac{d}{dt} \int_M f^p = p \int_M f^{p-1} f' - \int_M HF f^p,
\end{equation}
and
\begin{equation}
\begin{split}
\label{estimates-ddtf}
\frac{d}{dt} \int_M f^p + \tfrac{1}{2} p(p-1) \int_M F^{ij} f_i f_j f^{p-2} + \epsilon^2 p \int_M \abs{DA}^2 F^{- \alpha} f^{p-1}&\\
&\msp[-350] \leq \sigma p \int_M F^{ij} h_{ki} h^k_j f^p + 4Cp \int_M f^p.
\end{split}
\end{equation}
By choosing
\begin{equation}
c_0 = \tfrac{1}{4} c, \quad \sigma \leq \min (\epsilon^3 p ^{-1/2} {c_0}^{-1}, \sigma_0), \quad \delta = \epsilon p^{-1/2},
\end{equation}
and by using (\ref{ellip-int-ineq}), the right-hand side of inequality (\ref{estimates-ddtf}) can be estimated from above by
\allowdisplaybreaks[4]
\begin{equation}
\label{estimates-ddtf-part2}
\begin{split}
& \quad \epsilon p^{1/2} c_0^{-1} \{ \epsilon^2 \int_M F^{ij} h_{ki} h^k_j f^p \}+ 4Cp \int_M f^p \\
&\leq \epsilon p^{1/2} c_0^{-1} \{ \delta^{-1}c(p-1) +c \} \int_M F^{ij} f_i f_j f^{p-2}\\
&  \quad + \epsilon p^{1/2} c_0^{-1} \{ \delta c (p-1) + c\} \int_M \abs{DA}^2 F^{-\alpha} f^{p-1} +
\{ 2C \epsilon p^{1/2} c_0^{-1} + 4Cp\} \int_M f^p\\
&= c_0^{-1} \{ p(p-1)c + \epsilon p^{1/2} c\} \int_M F^{ij} f_i f_j f^{p-2}\\
& \quad + c_0^{-1} \{ \epsilon^2 (p-1)c
 + \epsilon p^{1/2} c\} \int_M \abs{DA}^2 F^{-\alpha} f^{p-1} + \{ 2C \epsilon p^{1/2} c_0^{-1} + 4Cp\} \int_M f^p\\
&\leq \tfrac{1}{2}p(p-1)  \int_M F^{ij} f_i f_j f^{p-2} + \tfrac12 \epsilon^2 (p-1) \int_M \abs{DA}^2 F^{-\alpha} f^{p-1}
+ 5C p \int_M f^p.
\end{split}
\end{equation}
From (\ref{estimates-ddtf}), (\ref{estimates-ddtf-part2}) we conclude that
\begin{equation}
\frac{d}{dt} \int_M f^p \leq 5Cp \int_M f^p,
\end{equation}
and the Gronwall's lemma leads to
\begin{equation}
\int_M f^p \leq {\int_M f^p } \big|_{t=0} \cdot \exp (5CpT^*),
\end{equation}
\begin{equation}
\norm{f}_p = \left( \int_M f^p \right)^\frac{1}{p} \leq e^{5CT^*} (\abs{M_0}+1) \sup_{0 \leq \sigma \leq 1/2} \sup_{M_0} f_\sigma.
\end{equation}
\end{proof}
To proceed further, we use the Stampacchia iteration scheme as in the Huisken's paper \cite[Theorem 5.1]{Hui84},
as well as \cite[Theorem 5.1]{Hui86}.
Note that $\mathbb{H}^{n+1}$ is simply connected and has constant sectional curvature $K_N = -1$,
thus the Sobolev inequality in \cite[Theorem 2.1]{Hof74} has the form
\begin{lem}
\label{sobolev-inequality}
Let $v$ be a nonnegative Lipschitz function on $M$, then there exists a constant $c=c(n)>0$, such that
\begin{equation}
(\int_M \abs{v}^{\frac{n}{n-1}})^{\frac{n-1}{n}} \leq c \{ \int_M \abs{Dv} + \int_M H \abs{v}\}.
\end{equation}
\end{lem}
Corresponding to \cite[Theorem 8.7]{Ger13}, we have
\begin{thm}
\label{decay-difference-principal-curvatures}
Let $F$ be strictly concave or $F = \frac{1}{n} H$, then there exist constants $\delta>0$ and $c_0>0$, such that
\begin{equation}
\abs{A}^2 - nF^2 \leq c_0 F^{2-\delta}.
\end{equation}
\end{thm}
\begin{proof}
As in the proof of \cite[Theorem 5.1]{Hui84} let $f_{\sigma,k}= \max (f_\sigma-k,0)$ for all $k \geq k_0 = \sup_{M_0} f_\sigma$ and denote
by $A(k)$ the set where $f_\sigma>k$. We obtain with $v= f_{\sigma,k}^{p/2}$ for $p \geq 4c \epsilon^{-2}$,
\begin{equation}
\begin{split}
\frac{d}{dt} \int_{A(k)} v^2 + \int_{A(k)} |Dv|^2 &\leq \sigma p \int_{A(k)} H^2 f_\sigma^p + 5Cp \int_{A(k)} f_\sigma^p\\
&\leq C(p) \int_{A(k)} H^2 f_\sigma^p.
\end{split}
\end{equation}
By applying Lemma \ref{sobolev-inequality} we can bound $f_\sigma$ for $\sigma$ small as in the proof of \cite[Theorem 5.1]{Hui84}.
The case $F = \frac{1}{n} H$ is proved in \cite[Lemma 5.1]{Hui86}.
\end{proof}
\begin{lem}
Let $F$ be strictly concave or $F= \frac{1}{n} H$ and $\tilde{M}(\tau)$ be the rescaled hypersurfaces,
then there are constants $c, \delta>0$ such that
\begin{equation}
\label{integral-estimates-D-tilde-A}
\int_{\tilde{M}} |D \tilde{A}|^2 \leq c e^{-\delta \tau} \quad \forall \tau_0 \leq \tau < \infty,
\end{equation}
where
\begin{equation}
\tau_0 = -\log \Theta(0,T^*), \quad |D \tilde{A}|^2 = \Theta^2 g^{ij} h^k_{l;i} \Theta h^l_{k;j} \Theta.
\end{equation}
\end{lem}
\begin{proof}
Choose
\begin{equation}
f = F^{-2} \{ \abs{A}^2 - nF^2\}.
\end{equation}
From Theorem \ref{decay-difference-principal-curvatures} we infer
\begin{equation}
f \leq c_0 F^{-\delta} \leq c \Theta^{\delta} = c e^{- \delta \tau} \quad \forall \tau \geq \tau_0,
\end{equation}
and from Theorem \ref{regularity-rescaled-function} we obtain
\begin{equation}
\label{estimates-DmA}
|D^m A| \leq c |A| \quad \forall m \geq 1.
\end{equation}
Integrating inequality (\ref{ineq-f}) over $M$, using integraion by parts and using relation (\ref{estimates-DmA}), we infer
\begin{equation}
\label{integral-estimates-D-tilde-A-unrescaled}
2 \epsilon^2 \int_M |DA|^2 F^{-2} \leq c \int_{M} f.
\end{equation}
Hence (\ref{integral-estimates-D-tilde-A}) follows by rescaling (\ref{integral-estimates-D-tilde-A-unrescaled}).
\end{proof}
Using the same proof of \cite[Lemma 8.10]{Ger13} we have
\begin{lem}
There are positive constants $c$ and $\delta$ such that for all $\tau \geq \tau_0$
\begin{equation}
\tilde{F}_{\max} - \tilde{F}_{\min} \leq c e^{-\delta \tau} ,
\end{equation}
and
\begin{equation}
\label{decay-DF}
\norm{D \tilde{F}} \leq ce^{-\delta \tau}.
\end{equation}
\qed
\end{lem}
\begin{lem}
There are positive constants $c$ and $\delta$ such that for all $\tau \geq \tau_0$
\begin{equation}
|D \tilde{u}| \leq ce^{-\delta \tau},
\end{equation}
where
\begin{equation}
|D \tilde{u}|^2 = \sigma^{ij} \tilde{u}_i \tilde{u}_j.
\end{equation}
\end{lem}
\begin{proof}
As in the proof of \cite[Lemma 8.12]{Ger13}, we let
\begin{equation}
\varphi = \log \tilde{u}, \quad w = \tfrac{1}{2} |D \varphi|^2,
\end{equation}
then
\begin{equation}
\label{expression-varphi}
\varphi' = - e^{- \varphi} \tilde{F} \Theta^{-1} \frac{\sinh \Theta}{\cosh \Theta} v +1.
\end{equation}
Differentiate now (\ref{expression-varphi}) with respect to $\varphi^k D_k$ we obtain
\begin{equation}
\begin{split}
w' = &2 e^{- \varphi} w \tilde{F} \Theta^{-1} \frac{\sinh \Theta}{\cosh \Theta} v
- e^{-\varphi} \tilde{F} \Theta^{-1} \frac{\sinh \Theta}{\cosh \Theta} v^{-1} \sinh^{-2} u \, u^2 w_k \varphi^k \\
 &+ R_1 + R_2,
\end{split}
\end{equation}
where
\begin{equation}
\begin{split}
R_1 &= - e^{-\varphi} \frac{\sinh \Theta}{\cosh \Theta} v F_k \varphi^k,\\
R_2 &=   e^{-\varphi} \tilde{F} \frac{\sinh \Theta}{\Theta \cosh \Theta} v^{-1} |D \varphi|^4 \sinh^{-3} u\{ u^3 \cosh u- u^2 \sinh u\} \geq 0.
\end{split}
\end{equation}
In view of (\ref{decay-DF}) $R_1$ decays exponentially. Thus the function
\begin{equation}
w_{\max} = \sup_{\tilde{M}(\tau)} w
\end{equation}
is Lipschitz and satisfies
\begin{equation}
w'_{\max} \geq 2 e^{- \varphi} w \tilde{F} \Theta^{-1} \frac{\sinh \Theta}{\cosh \Theta} v - c e^{-\delta \tau}
\end{equation}
for almost every $\tau \geq \tau_0$. Using the same argument as in \cite[Lemma 8.12]{Ger13} we conclude that
\begin{equation}
w_{\max} (\tau) \leq \tfrac{c}{\delta} e^{- \delta \tau} \quad \forall \tau \geq \tau_0.
\end{equation}
\end{proof}
The same arguments of \cite[Corollary 8.13]{Ger13} and the interpolation inequalities for the $C^m$-norms
 (cf. \cite[Corollary 6.2]{Ger11}) yield
\begin{thm}
\label{convergence-tilde-u-to-1}
Let $F$ be strictly concave or $F = \frac{1}{n} H$, then the rescaled function $\tilde{u}$ converges  in $C^{\infty}(\mathbb{S}^n)$
to the constant function $1$ exponentially fast. \qed
\end{thm}

\begin{lem}
Let $F$ be strictly concave or $F = \frac{1}{n} H$, then there exist positive constants $c$ and $\delta$ such that
\begin{equation}
|\tilde{F} (\tau, \cdot) - 1| \leq c e^{- \delta \tau} \quad \forall \tau \geq \tau_0.
\end{equation}
\end{lem}
\begin{proof}
Observe that for $\tau_1$ sufficiently large we have
\begin{equation}
\left| \frac{\sinh \Theta}{\cosh \Theta} - \Theta \right| \leq c \Theta^2 \quad \forall \tau \geq \tau_1.
\end{equation}
The rest of the proof is identical to \cite[Lemma 8.16]{Ger13}.
\end{proof}
\section{Inverse curvature flows}
\label{section:inverse-curvature-flows}
Let $M(t)$ be the flow hypersurfaces of the direct flow in $\mathbb{H}^{n+1}$ and write $M(t)$ as graphs $M(t) = \textrm{graph} \, u(t, \cdot).$
with respect to the geodesic polar coordinates centered in the point where the direct flow shrinks to.
By applying an isometry we may assume that the point $x_0$ is the Beltrami point.
The polar hypersurfaces $M(t)^*$ are the flow hypersurfaces of the corresponding inverse curvature flow in the de Sitter space. Write
$M(t)^* = \textrm{graph} \, u^*(t, \cdot)$ over $\mathbb{S}^n$.
\begin{lem}
The functions $u, u^*$ satisfy the relations
\begin{equation}
\label{relation-u-u*-1}
u_{\max} = - u^*_{\min} \quad \forall t \in [t_\delta, T^*),
\end{equation}
\begin{equation}
\label{relation-u-u*-2}
u_{\min} = - u^*_{\max} \quad \forall t \in [t_\delta, T^*).
\end{equation}
\end{lem}
\begin{proof}
We use the relation \cite[(10.4.65)]{Ger06}
\begin{equation}
\tilde{x}^0 = \frac{r}{\sqrt{1-r^2}},
\end{equation}
and note that by comparing \cite[(10.2.5)]{Ger06} and the metric in the eigentime coordinate system in $N$  (\ref{eigen-time-coordinates})
we infer that
\begin{equation}
\cosh^2 u^* = 1+ \abs{\tilde{x}^0}^2.
\end{equation}
From (\ref{relation-rho-r}) we infer that
\begin{equation}
r = \tanh u.
\end{equation}
Since we have switched the light cone such that the uniformly convex slices are contained in $\{ \tau <0 \}$, we deduce that
\begin{equation}
u^* = - \mathrm{arcsinh} (\tilde{v} \sinh u) = -\mathrm{arcsinh}  \tilde{\chi}.
\end{equation}
In a point where $u^*$ attains its minimum, there holds $v=1$ in view of Lemma \ref{chi-i-u-i}. Thus $u = -u^*$ and $u$ attains its maximum in
such a point. This proves (\ref{relation-u-u*-1}). The proof of (\ref{relation-u-u*-2}) is similar.
\end{proof}
\begin{cor}
There exists a positive constant $c$ such that
\begin{equation}
- c \leq w \equiv  u^* \Theta^{-1} \leq -c^{-1} \quad \forall t \in [ t_\delta, T^*).
\end{equation}
\qed
\end{cor}
Define $\vartheta (u) = \cosh (u)$ and $\bar{g}_{ij} = \vartheta^2 \sigma_{ij}$.
We prove in the following that $w$ is uniformly bounded in $C^{\infty}(\mathbb{S}^n)$.
For simplicity, we write in the following $u$ instead $u^*$ for the graphs of the flow hypersurfaces in the de Sitter space.
The proof of $C^1$-estimates of $w$ is similar to \cite[Theorem 2.7.11]{Ger06}.
\begin{lem}
There exists a positive constant $c$ such that
\begin{equation}
\abs{D w}^2  \equiv \sigma^{ij} w_i w_j \leq c \quad \forall t \in [t_\delta, T^*).
\end{equation}
\end{lem}
\begin{proof}
Since
\begin{equation}
\norm{Du}^2 \equiv g^{ij} u_i u_j = v^{-2} \bar{g}^{ij} u_i u_j \equiv v^{-2} \abs{Du}^2,
\end{equation}
we first estimate $\norm{Du} \Theta^{-1}$.
Let $\lambda$ be a real parameter to be specified later and define
\begin{equation}
G = \tfrac{1}{2} \log (\norm{Du}^2 \Theta^{-2}) + \lambda u \Theta^{-1}.
\end{equation}
There is $x_0 \in \mathbb{S}^n$ such that
\begin{equation}
G(x_0) = \sup_{\mathbb{S}^n} G,
\end{equation}
and thus in $x_0$
\begin{equation}
0 = G_i = \norm{Du}^{-2} u_{ij} u^j + \lambda u_i \Theta^{-1},
\end{equation}
where the covariant derivatives are taken with respect to $g_{ij}$ and
\begin{equation}
u^i = g^{ij} u_j = v^{-2} \bar{g}^{ij} u_j.
\end{equation}
Since
\begin{equation}
h_{ij} v^{-1} = - u_{ij} - \dot\vartheta \vartheta \sigma_{ij},
\end{equation}
we infer that
\begin{equation}
\begin{split}
\lambda \norm{Du}^{-4} \Theta^{-4} &= - u_{ij} u^i u^j \Theta^{-3}\\
& = v^{-1} h_{ij} u^i u^j \Theta^{-3} + \dot\vartheta \vartheta \sigma_{ij} u^i u^j \Theta^{-3}.\\
\end{split}
\end{equation}
By considering the dual flow in the hyperbolic space, we conclude that $h_{ij}>0$. Furthermore,
\begin{equation}
\begin{split}
\dot\vartheta \vartheta \sigma_{ij} u^i u^j \Theta^{-3} &= (\dot\vartheta \Theta^{-1}) \vartheta^{-1} v^{-2} \norm{Du}^2 \Theta^{-2}.
\end{split}
\end{equation}
By applying \cite[Theorem 2.7.11]{Ger06} directly, we conclude that $v^{-2}$ is uniformly bounded. Note $\dot\vartheta \Theta^{-1} \leq c$.
Let $c_0$ be an upper bound for $(\dot\vartheta \Theta^{-1}) \vartheta^{-1} v^{-2}$ and by choosing $\lambda < -c_0$ we conclude that
$\norm{Du} \Theta^{-1}$ can not be too large in $x_0$. Thus $\norm{Du} \Theta^{-1}$ is uniformly bounded from above. We conclude that
\begin{equation}
\sigma^{ij} w_i w_j = \norm{Du}^2 \Theta^{-2} \theta^2 v^2
\end{equation}
is uniformly bounded.
\end{proof}
\begin{lem}
There exists a positive constant $c$ such that for all $m \geq 2$
\begin{equation}
|D^m w|^2  \leq c \quad \forall t \in [t_\delta, T^*).
\end{equation}
\end{lem}
\begin{proof}
Let $(\hat{h}^{ij}) = (h_{ij})^{-1}$ be the inverse of the second fundamental form in $\mathbb{H}^{n+1}$ and $\tilde{h}_{ij}$ the second fundamental
form in $N$. We consider the mixed tensor
\begin{equation}
\hat{h}_i^j = g_{ik} \hat{h}^{kj}, \quad \tilde{h}_i^j = \tilde{g}^{kj} \tilde{h}_{ki},
\end{equation}
where $g_{ij}$ and $\tilde{g}_{ij} = h_i^k h_{kj}$ are the metrics of hypersurfaces in $\mathbb{H}^{n+1}$ resp. $N$.
From the relation
\begin{equation}
\tilde{\kappa}_i = \kappa_i^{-1},
\end{equation}
we infer that
\begin{equation}
\tilde{h}_i^j  = \hat{h}_i^j.
\end{equation}
From Theorem \ref{regularity-rescaled-function} we infer that $h_i^j \Theta$ are uniformly bounded in $C^{\infty} (\mathbb{S}^n)$ and due to
Lemma \ref{rescaled-curvature-bounded-from-above} and Corollary \ref{rescaled-curvature-bounded-from-below} there are constants $c_1, c_2>0$ such that
\begin{equation}
0 < c_1 \delta_i^j \leq h_i^j \Theta \leq c_2 \delta_i^j,
\end{equation}
and thus $\tilde{h}_i^j \Theta^{-1} = \hat{h}_i^j \Theta^{-1}$, as the inverse of $h_i^j \Theta$, are uniformly bounded in $C^{\infty} (\mathbb{S}^n)$.
We switch now our notation by considering the quantities in $N$ without writing a tilde.
Denote the covariant derivatives with respect to $\bar{g}_{ij}$ resp. $\sigma_{ij}$ by a semiconlon resp. a colon.
In view of \cite[Remark 1.6.1, Lemma 2.7.6]{Ger06} we have
\begin{equation}
\begin{split}
v^{-1} h_{ij} &= -v^{-2} u_{;ij} - \dot\vartheta \vartheta \sigma_{ij}\\
& = -v^{-2} \{ u_{:ij} - \tfrac{1}{2} \bar{g}^{km} \left( (\vartheta^2)_j \sigma_{mi} + (\vartheta^2)_i \sigma_{mj} -
(\vartheta^2)_m \sigma_{ij} \right) u_k \}
- \dot\vartheta \vartheta \sigma_{ij}.\\
\end{split}
\end{equation}
Therefore,
\begin{equation}
\label{relation-u_ij-h_ij}
u_{:ij} = -v h_{ij} + 2 \vartheta^{-1} \dot\vartheta u_i u_j - \vartheta \dot\vartheta \sigma_{ij}.
\end{equation}
By considering the dual flow in hyperbolic space, we infer that
\begin{equation}
\abs{A} \Theta^{-1} \leq c,
\end{equation}
and note that
\begin{equation}
\bar{g}^{ij} \leq \bar{g}^{ij} + v^{-2} \check{u}^i \check{u}^j = g^{ij},
\end{equation}
where
\begin{equation}
\check{u}^i = \bar{g}^{ij} u_j,
\end{equation}
we conclude that
\begin{equation}
\sigma^{ik} \sigma^{jl} h_{ij} h_{kl} \leq c |A|^2.
\end{equation}
In view of $\dot\vartheta \Theta^{-1} \leq c$ we conclude that $|D^2 w|^2$ is uniformly bounded.\\
Contract $(\ref{relation-u_ij-h_ij})$ with $g^{ij}$ we conclude further
\begin{equation}
\label{elliptic-eq-w}
- g^{ij} w_{:ij} -\vartheta^{-3} \dot\vartheta \Theta v^{-2} |Dw|^2 + v H \Theta^{-1} + n \vartheta^{-1} \dot\vartheta \Theta^{-1} =0.
\end{equation}
Since $v$ is uniformly bounded, (\ref{elliptic-eq-w}) is a uniformly elliptic equation in $w$ with bounded coefficients.
A bootstrapping procedure with Schauder theory yields for all $m \in \mathbb{N}$
\begin{equation}
|w|_{m, \mathbb{S}^n} \leq c_m \quad \forall t \in [0, T^*).
\end{equation}
\end{proof}
From Lemma \ref{convergence-tilde-u-to-1} and preceding results in Section \ref{section:inverse-curvature-flows} we conclude
\begin{thm}
Let the geodesic polar coordinates $(\tau, \xi^i)$ of $N$ be specified in Section $2$. Represent the inverse curvature flow
$(\ref{inverse-flow-N})$ in $N$ as graphs over $\mathbb{S}^n$,
$M(t)^* = \textrm{graph} \, u^*(t, \cdot)$,
where the curvature function $\tilde{F}$ satisfies the assumption $\ref{general-assmption}$. Then $u^*$ converges to the constant function $0$
in $C^\infty (\mathbb{S}^n)$. The rescaled function $w = u^* \Theta^{-1}$ are uniformly bounded in $C^{\infty} (\mathbb{S}^n)$. When
the curvature function $F$ of the corresponding contracting flow is strictly concave or $F = \frac{1}{n} H$,
then $w(\tau, \cdot)$ converges in $C^{\infty} (\mathbb{S}^n)$ to
the constant function $-1$ exponentially fast. \qed
\end{thm}
\enlargethispage{1cm}
\bibliographystyle{hamsplain}
\bibliography{mrabbrev,publications}



\end{document}